\documentclass[12pt]{amsart}

\usepackage{amssymb,amssymb}
\usepackage{amsthm}
\usepackage{verbatim,here}
\usepackage{graphicx}
\usepackage{epstopdf}
\usepackage{floatflt}
\usepackage[top=3cm,bottom=3cm,left=2.5cm,right=2.5cm]{geometry}

\newtheorem{thm}[equation]{Theorem}

\newtheorem{lem}[equation]{Lemma}

\newtheorem{prop}[equation]{Proposition}
\newtheorem{conjecture}[equation]{Conjecture}

\newtheorem{rem}[equation]{Remark}
\theoremstyle{definition}
\newtheorem{defn}[equation]{Definition}
\theoremstyle{remark}

\newtheorem{nonsec}[equation]{}

\DeclareMathOperator{\dist}{dist}
\newcommand{\half}{\textstyle{1\over 2}}
\renewcommand{\Im}{{ \rm Im}\,}
\renewcommand{\Re}{{ \rm Re}\,}
\makeatletter
	
	\@addtoreset{equation}{section}
\makeatother
\makeatletter
\renewcommand{\subsection}{%
    \stepcounter{subsection}
    \addtocounter{equation}{+1}
    \setcounter{subsection}{\value{equation}}
    \medskip
    \noindent{{\bfseries \arabic{section}.\arabic{subsection}.\ }}}
\renewcommand{\thesubsubsection}{%
    \@arabic\c@section.\@arabic\c@subsection.\@arabic\c@subsubsection}
\makeatother
\usepackage{color,colortbl}
\newcounter{minutes}
\divide\time by 60
\newcounter{hours}
\multiply\time by 60 \addtocounter{minutes}{-\time}

\renewcommand{\thefootnote}{\number_style{footnote}}
\begin{document}
\def\thefootnote{}

\title[]{Barrlund's distance function and quasiconformal maps}
\author[M. Fujimura]{Masayo Fujimura}
\address{Department of Mathematics, National Defense Academy of Japan, Japan}

\author[M. Mocanu]{Marcelina Mocanu}
\address{Department of Mathematics and Informatics, Vasile Alecsandri
         University of Bacau, Romania }
\author[M. Vuorinen]{Matti Vuorinen}
\address{Department of Mathematics and Statistics, University of Turku,
         Turku, Finland}
\date{}

\maketitle

\begin{abstract}
  Answering a question about triangle inequality suggested
  by R. Li, A. Barrlund \cite{b}
  introduced a distance function which is a metric on a subdomain
  of ${\mathbb R}^n\,.$ We study this Barrlund metric
  and give sharp bounds for it in terms of other metrics
  of current interest. We also prove sharp distortion results for
  the Barrlund metric under quasiconformal maps.
\end{abstract}

\footnotetext{\texttt{{\tiny File:~\jobname .tex, printed: \number\year-%
\number\month-\number\day, \thehours.\ifnum\theminutes<10{0}\fi\theminutes}}}
\makeatletter

\makeatother

\section{Introduction}
\label{section1}
\setcounter{equation}{0}

For a given domain $G\subset \mathbb{R}^{n}$ with $G\neq \mathbb{R}^{n}$,
for a number $p\ge 1\,$, and for points $z_1,z_2 \in G\,,$ let
\begin{equation}  \label{eq:bdef}
  b_{G,p}(z_1,z_2)= \sup_{z \in \partial G}
  \frac{ |z_1-z_2|}{\sqrt[p]{|z_1-z|^p+|z-z_2|^p}} \, .
\end{equation}
A. Barrlund \cite{b}
\footnote[1]{Anders Barrlund 1962-2000 was a Swedish mathematician
and \cite{b} was his last paper.}
studied this expression for the case $G= \mathbb{R}^n
\setminus \{0\}$ and proved, answering a question of R.-C. Li \cite{li}, that
it is a metric. These facts motivated, in part,
P. H\"ast\"o's  papers \cite{h1,h2}, where he proved that  $b_{G,p}$
is a metric in a general domain and studied also some other metrics.

The {\sl triangular ratio metric} $s_G$ of a given domain
$G \subset {\mathbb{R}}^n$  defined as follows
\begin{equation}  \label{eq:sdef}
  s_G(z_1,z_2)= \sup_{z \in \partial G} \frac{|z_1-z_2|}{|z_1-z|+|z-z_2|} \, ,
  z_1,z_2 \in G\,,
\end{equation}
was recently studied in  \cite{chkv,hvz}. 
As shown in \cite{hvz}, this metric is closely related to the 
quasihyperbolic metric \cite{hkvz,go,wv} and several other metrics of 
current interest \cite{himps,rtz,pac,gp}.

We study {\sl the Barrlund metric} $b_{G,p}\,$ and compare it to
$s_G = b_{G,1}\,.$
For the cases of a ball or a half-plane we give in our main 
theorems \ref{cor:barrdist} and \ref{binH} explicit formulas for  $ b_{G,2}\,.$
To this end, we first recall some properties of $s_G\,.$
By compactness, the suprema in \eqref{eq:bdef} and \eqref{eq:sdef} are attained.
If $G$ is convex, it is simple to see that the extremal point $z_0$ for
\eqref{eq:sdef} is a point of contact of the boundary with an ellipse
contained in $G\,$  with foci at $z_1,z_2\,.$

We prove the following sharp inequality between the above two metrics.

\begin{thm} \label{sbThm} Let $G$ be a domain in
  $ {\mathbb{R}}^n\,$ and  let $p\ge 1\,.$ Then for all points $z_1,z_2 \in G\,$
  \begin{equation*}
    s_{ G }(z_1,z_2) \le b_{ G ,p}(z_1,z_2) \le 2^{1-1/p} s_{ G}(z_1,z_2)\,.
  \end{equation*}
\end{thm}

Clearly, this inequality holds as an identity if $p=1\,.$
But perhaps more interesting is that the right hand side holds
as an equality for all $p\ge 1$ if $G=\{z \in \mathbb{C}: \Im(z)>0\} \,,$
and $ z_1, z_2 \in G$  with ${\rm Im}(z_1)  = {\rm Im}(z_2) \,.$

The metric $s_{ {\mathbb{D}}}$ is also connected with a classical
problem of optics.
The well-known Ptolemy-Alhazen problem reads \cite{s}:
"Given a light source and a spherical mirror, find the point on 
the mirror where the light will be reflected to the eye of an observer."
We consider now the following two-dimensional version of the problem when
two points $z_1,z_2$ are in the unit disk
$\mathbb{D}=\{ z \in \mathbb{C}: |z|<1\}$ and its
 circumference $\partial \mathbb{D} = \{ z \in \mathbb{C}:
|z|=1\}$ is a reflecting curve.
The problem is to find all points $u \in \partial \mathbb{D}$ such that
\begin{equation}  \label{theptu}
  \measuredangle (z_{1},u,0)=\measuredangle (0,u,z_{2})\, .
\end{equation}
 Here $\measuredangle (z,u,w)$ denotes the
radian measure in $(-\pi ,\pi ]$ of the oriented angle with initial side
$[u,z]$ and final
side $[u,w]$ . This condition says that the angles of incidence and reflection
are equal, a light ray from $z_1$ to $u$ is reflected at $u$ and goes
through the point $z_2\,.$

The equality \eqref{theptu} shows that the ellipse with foci $z_{1}$,
$z_{2}$, passing
through $u$, is tangent at $u$ to the unit circle.
A point $u=e^{i\theta _{0}}\in \partial \mathbb{D}$ satisfies \eqref{theptu}
if and only if $\theta _{0}$ is a critical point of 
$ f\left( \theta \right) :=\left\vert e^{i\theta }-z_{1}\right\vert 
  +\left\vert e^{i\theta}-z_{2}\right\vert $, 
$ \theta \in \mathbb{R}$.
Note that $f^{\prime }\left( \theta \right) ={\rm {Im}}\left( z\overline{w}%
\right) $, where $z=e^{i\theta }$ and $w=\frac{e^{i\theta }-z_{1}}{%
\left\vert e^{i\theta }-z_{1}\right\vert }+\frac{e^{i\theta }-z_{2}}{%
\left\vert e^{i\theta }-z_{2}\right\vert }$, therefore $f^{\prime }\left(
\theta \right) =0$ if and only if the radius of the unit circle terminating
at $z$ is the bisector of the angle formed by segments joining $z_{1},z_{2}$
to $z$.

Now for the case of the unit disk $G= \mathbb{D}$ and
$z_1,z_2 \in \mathbb{D}\,$ and the extremal point
$ z_0 \in \partial \mathbb{D}\,,$ for the definition \eqref{eq:sdef},
the connection between the triangular ratio metric
\begin{equation*}
  s_{\mathbb{D}}(z_1, z_2) = \frac{|z_1- z_2|}{ |z_1- z_0| + |z_2-z_0| }
\end{equation*}
and the Ptolemy-Alhazen  problem is clear: $u$ =$z_0$ satisfies
\eqref{theptu}. This connection was recently pointed out in \cite{fhmv}.

\begin{thm}[\cite{fhmv}] \label{Fujimurathm}
  \quad The point $u $ in \eqref{theptu} is given as a solution of the
  equation
  \begin{equation}  \label{eq:equation}
     \overline{z_1}\overline{z_2}u^4-(\overline{z_1}+\overline{z_2})u^3
     +(z_1+z_2)u-z_1z_2=0.
  \end{equation}
\end{thm}

This quartic equation can be solved by symbolic computation programs.
This method was used in  \cite{fhmv} to compute
the values of $s_{\mathbb{D}}(z_1,z_2)\,.$

We  also study the limiting case $p=\infty$ of the Barrlund metric.
As pointed out by
P. H\"ast\"o \cite{h1},
it was proved by D. Day in a short note \cite{d} that the
$ p $-relative distance with $p=\infty$ is a metric in $G$, for
$G= \mathbb{R}^n\setminus \{0\}$.

We conclude our paper by studying the behavior of the Barrlund distance
under M\"obius transformations and
quasiconformal mappings defined on the upper half plane $ \mathbb{H}$
and prove the following theorem.

\begin{thm} \label{myBarrQc}
  Let $f:  \mathbb{H} \to  \mathbb{H}$ be a $K$-quasiconformal map and
  $z_1,z_2 \in \mathbb{H}\,.$ Then for $p\ge 1$
  \begin{equation*}
    b_{\mathbb{H},p}(f(z_1), f(z_2))
     \le 2^{1-1/p} 4^{1-1/K} {b_{\mathbb{H},p}(z_1,z_2) }^{1/K} \,.
  \end{equation*}
\end{thm}

Observe that this theorem is sharp.

We also formulate two conjectures.

\begin{rem} {\rm
  After the publication of \cite{fhmv},
  we have learned more about the history of the Ptolemy-Alhazen problem:
  e.g. the book of A.M. Smith \cite{s} provides a historical account of
  Alhazen's work on optics.
  Dr. F.G. Nievinski has kindly informed us about the papers
  of P.M. Neumann \cite{n} and J.D. Smith \cite{sj}, which also
 study this problem.
  The equation \eqref{eq:equation} appears also in \cite[(1), p. 525]{n}
  and \cite[p.194 line 1]{sj}. Note that
  in \cite{fhmv} we study this topic from a different point of view.
}
\end{rem}

\section{Preliminaries}

We recall the definition of  the hyperbolic
distance $\rho_{\mathbb{D}}(z_1,z_2)$ between two points $z_1,z_2 \in
\mathbb{D} $ \cite[Thm 7.2.1, p. 130]{be}:
\begin{equation}  \label{eq:tro}
  \tanh {\frac{\rho_{\mathbb{D}}(z_1,z_2)}{2}}
     =\frac{|z_1-z_2|}{\sqrt{|z_1-z_2|^2+(1-|z_1|^2)(1-|z_2|^2)}}\,.
\end{equation}
The triangular ratio
metric can be estimated in terms of the hyperbolic metric as follows.
By \cite[2.16]{hvz} for $z_1,z_2 \in \mathbb{D}$
  \begin{equation}  \label{eq:hvz216}
    \tanh {\frac{\rho_{\mathbb{D}}(z_1,z_2)}{4}}
    \le s_{\mathbb{D}} (z_1,z_2)
    \le \tanh {\frac{\rho_{\mathbb{D}}(z_1,z_2)}{2}} \,.
  \end{equation}

\begin{conjecture}
  \label{trineq} The function
  \begin{equation*}
    \mathrm{artanh}\,s_{\mathbb{D}}(z_1,z_2)
  \end{equation*}
  satisfies the triangle inequality.
\end{conjecture}

We have checked this conjecture  using the aforementioned formula
\cite{fhmv} for $s_{\mathbb{D}}(z_1,z_2)$ based on Theorem \ref{Fujimurathm}
and found no
counterexamples. Experiments also show that for points $0<r<s<t<1$ we
have the following addition formula
\[
    \mathrm{artanh}\,s_{\mathbb{D}}(r,t)
   =\mathrm{artanh}\,s_{\mathbb{D}}(r,s)+\mathrm{artanh}\,s_{\mathbb{D}}(s,t)\,
\]
and this equality statement also follows from formula \eqref{sOnDia} below.

\medskip

Let $G \subset {\mathbb{R}}^n\,$ be a proper open subset of 
${\mathbb{R}}^n\,.$ As in \cite{chkv}, 
we define the point pair function $p_G$ as follows
for $z_1,z_2 \in G\,:$
\[
  p_G(z_1,z_2)=\frac{|z_1-z_2|}{\sqrt{|z_1-z_2|^2+4\,d_G(z_1)\,d_G(z_2)}}\,,
\]
where $d_{G}(x)=\dist (x, \partial G)\,.$
By \cite[Lemma 3.4 (1)]{chkv} if $G$ is
convex and  $z_1,z_2\in G \subset\mathbb{R}^n\,,$  then
\begin{equation}  \label{slep}
  s_G(z_1,z_2)\leq p_G(z_1,z_2)\,.
\end{equation}

\begin{thm}
  \label{slem} If $z_1,z_2\in \mathbb{D}$,
  \begin{equation} \label{slem2}
    s_{\mathbb{D}}(z_1,z_2)\leq m_{\mathbb{D}}(z_1,z_2) :=\frac{|z_1-z_2|}{2-|z_1+z_2|}.
  \end{equation}
  Here equality holds if and only if $z_1, 0, z_2$ are collinear.
\end{thm}

\begin{proof}
Fix $z_1, z_2\in \mathbb{D}$, and let $u\in \partial \mathbb{D}\,.$
Then by the triangle inequality we have
\begin{equation*}
  \frac{|z_1-z_2|}{|z_1-u|+|z_2-u|}\leq \frac{|z_1-z_2|}{|2u-(z_1+z_2)|}
      \leq \frac{|z_1-z_2|}{\big|2|u|-|z_1+z_2|\big|}
      = \frac{|z_1-z_2|}{2-|z_1+z_2|}\,.
\end{equation*}
Hence the inequality follows. The equality statement follows from the
equality statement for the triangle inequality.
\end{proof}

Note that the equality statement in \eqref{slem2} implies for  $0<r<s<1$ that
\begin{equation} \label{sOnDia}
  \mathrm{artanh}\,s_{\mathbb{D}}(r,s)= \frac{1}{2} \log \frac{1-r}{1-s} \,.
\end{equation}

\begin{rem} {\rm
  The inequalities \eqref{slep} and \eqref{slem2} are not comparable.
  We always have
  \[
        s_{\mathbb{D}}(z_1,z_2) \leq p_{\mathbb{D}} (z_1,z_2)
         \le \tanh {\frac{\rho_{\mathbb{D}}(z_1,z_2)}{2}} <1\,.
  \]
  Sometimes
  $p_{\mathbb{D}} (z_1,z_2)> m_{\mathbb{D}} (z_1,z_2)\,.$ On the other hand the
  function $m_{\mathbb{D}}$ is unbounded.
  Finally, for $r,t \in (0,1)$ we have
  $p_{\mathbb{D}} (r,t)= m_{\mathbb{D}} (r,t)\,.$
  It is easily seen that
  $ m_{\mathbb{D}} (t,it)> m_{\mathbb{D}}(0,t)+  m_{\mathbb{D}}(0,it) $ for
  $t \in (0.85, 1) $
  and hence $ m_{\mathbb{D}}$ is not a metric.
}
 \end{rem}

\section{On Barrlund's metric} \label{section3}

In this section we will give explicit formulas for the Barrlund metric
\eqref{eq:bdef} when $p=2$ and the domain is either the unit disk or
the upper half plane and study some properties of the Barrlund metric for
$1\leq p\leq \infty $.

\subsection{\bf Basic properties of the Barrlund metric.} \label{sec:barrprop}

  Suppose that $G $ is a proper subdomain of the complex
  plane and $p \ge 1\,.$ Because $s_G(z_1,z_2) = b_{G,1}(z_1,z_2)$ for all
  $z_1,z_2 \in G\,,$ it is natural to expect that some properties of $s_G$
  might have a counterpart also for $b_{G,p}\,, \ p>1\,.$
   We list a few
  immediate observations and recall first the notion of midpoint convexity.

\begin{defn} \label{midpConv} \cite[p.60]{bv}
  A domain $G \subset {\mathbb{R}^n}$ is {\it midpoint convex} if for 
  $x,y \in G$ also
  the midpoint $(x+y)/2 \in G\,.$
\end{defn}

Obviously, every convex set is midpoint convex. 
If a midpoint convex set in $\mathbb{R}^{n}$ is closed or is open, 
then the set is convex. In particular, every midpoint
convex domain is also convex.

\begin{enumerate}
    \item[(1)] \label{item:barr1}
           {\textrm{If $\lambda >0, a \in \mathbb{C}\,,$ and $%
           h(z) = \lambda z + a\,,$ then $b_{G,p}$ is invariant
           under $h\,,$ i.e. for all $z_1,z_2 \in G\,,$
           \begin{equation*}
              b_{h(G),p}(h(z_1),h(z_2)) = b_{G,p}(z_1,z_2) \,.
            \end{equation*}
            }}
    \item[(2)] \label{item:barr2}
          {\textrm{$b_{G,p}$ is monotone with respect to the
           domain: If $G_1$ is a midpoint convex subdomain of $G$ and
           $z_1,z_2 \in G_1\,,$ then $b_{G,p}(z_1,z_2) \le b_{G_1,p}(z_1,z_2)\,,$ 
           see Lemma \ref{LemMidP}.
           In particular, if $ G $ is midpoint convex,
           \begin{equation*}
              b_{G,p}(z_1,z_2) \ge \sup \{ b_{\mathbb{C} \setminus \{z\},p }(z_1,z_2) :
                   z \in\partial G \}\,.
           \end{equation*}
           }}
    \item[(3)] \label{item:barr3}
            {\textrm{$b_{G,p}$ satisfies the triangle inequality,
             i.e. it is a metric.
             }}
  \end{enumerate}

\begin{rem} \label{CounterEx}
Replacing $\partial G$ by $\mathbb{R}^{n}\setminus G$ 
in Definition \eqref{eq:bdef} we obtain a modified
Barrlund function 
that is is monotone with respect to the domain.

We show here that for $p=2$ and $n=2$ 
the monotonicity with
respect to the domain (3) does not hold for all domains 
$G_{1}\subset G\subsetneq \mathbb{R}^{n}$.
\begin{enumerate}
  \item[(1)]
    We first observe that by elementary geometry (Stewart's theorem) for all 
    $x,y, w\in \mathbb{R}^{n}$
    \[
      | w-x| ^{2}+| w-y| ^{2}=2| w-\frac{1}{2}\left( x+y\right) |^{2}
         +\frac{1}{2}|x-y| ^{2}\,.
    \]
  \item[(2)]
     The formula in (1) implies 
     that for a domain $D\varsubsetneq \mathbb{R}^{n}$
     and for $x,y\in D$
     \[
        b_{D,2}(x,y)=\frac{| x-y| }{\sqrt{2d_{D}^{2}(\frac{1}{2}%
        \left( x+y\right) )+\frac{1}{2}| x-y| ^{2}}}\,.
     \]
  \item[(3)]
    For $a>0$ let  
      $S_{a} =\{ z \in \mathbb{C}:  \Re(z), \Im(z) \in (-a,a)\} $ be a square
     and $G= S_4 \setminus \overline{S}_1$
     and $G_1=  S_4 \setminus \overline{S}_2\,.$
     With $z_{1}=3$, $z_{2}=-3$ we have $z_{1},z_{2}\in G_{1}\subset G$, 
     but by part (2)
     \[
        \frac{6}{\sqrt{26}}=b_{G_{1},2}\left(
     z_{1},z_{2}\right) <b_{G,2}\left( z_{1},z_{2}\right) =\frac{6}{\sqrt{20}}\,.
     \]
  \end{enumerate}
\end{rem}

\begin{lem} \label{LemMidP}
  Let $1\leq p\leq \infty $. If $G_{1}\subset G\subsetneq \mathbb{R}^{n}$ are
  domains, such that $G_{1}$ is midpoint convex, then $b_{G_{1},p}(x,y)\geq
  b_{G,p}(x,y)$ for all $x,y\in G_{1}$.
\end{lem}

\begin{proof}
Fix $x,y\in G_{1}$. There exists $a=a(p)\in \partial G$ such
that
\begin{equation*}
  b_{G,p}(x,y)=\frac{\left\vert x-y\right\vert }{\sqrt[p]{\left\vert
  x-a\right\vert ^{p}+\left\vert y-a\right\vert ^{p}}}
\end{equation*}%
if $1\leq p<\infty $, respectively
\begin{equation*}
  b_{G,\infty }(x,y)=\frac{\left\vert x-y\right\vert }{\max \left\{ \left\vert
  x-a\right\vert ,\left\vert y-a\right\vert \right\} }.
\end{equation*}

Since $G_{1}$ is midpoint convex, $G_{1}$ contains 
$m=\frac{1}{2}\left(x+y\right) $. 
The intersection of the segment $\left[ m,a\right] $ with the
boundary $\partial G_{1}$ contains at least one point, which we denote by 
$d\,.$%

We prove that
\[
  \max \left\{ \left\vert x-d\right\vert ,\left\vert
  y-d\right\vert \right\} \leq \max \left\{ \left\vert x-a\right\vert,
  \left\vert y-a\right\vert \right\} 
\]
and that
\[
  \left\vert x-d\right\vert^{p}+\left\vert y-d\right\vert ^{p}
  \leq \left\vert x-a\right\vert^{p}+\left\vert y-a\right\vert ^{p}
\]
if $1\leq p<\infty \,.$

Then
\[
  b_{G_{1},\infty }(x,y) \geq 
   \frac{\left\vert x-y\right\vert}
     {\max\left\{\left\vert x-d\right\vert,\left\vert y-d\right\vert\right\}}
   \geq b_{G,\infty }(x,y)
\]
and
\[
   b_{G_{1},p}(x,y)\geq 
    \frac{\left\vert x-y\right\vert }
         {\sqrt[p]{\left\vert x-d\right\vert^{p}+\left\vert y-d\right\vert^{p}}}
    \geq b_{G,p}(x,y)
\]
if $1\leq p<\infty .$

Let $\lambda \in \lbrack 0,1)$ such that $d=(1-\lambda )a+\lambda m$. 
For every $z\in \mathbb{R}^{n}$, $\left( z-d\right) =\left( 1-\lambda \right)
\left( z-a\right) +\lambda \left( z-m\right) $, hence
\begin{equation}
  \left\vert z-d\right\vert \leq \left( 1-\lambda \right) \left\vert
  z-a\right\vert +\lambda \left\vert z-m\right\vert .  \label{convexcomb}
\end{equation}

If $p=\infty $, note that (\ref{convexcomb}) implies
\[
  \max \left\{\left\vert x-d\right\vert ,\left\vert y-d\right\vert \right\}
   \leq \left(1-\lambda \right) 
    \max \left\{ \left\vert x-a\right\vert ,\left\vert y-a\right\vert \right\}
    +\lambda 
  \max\left\{ \left\vert x-m\right\vert,\left\vert y-m\right\vert \right\}\,. 
\]
But
\begin{equation}
  \left\vert x-m\right\vert =\left\vert y-m\right\vert 
  =\frac{1}{2}\left\vert x-y\right\vert 
  \leq \frac{1}{2}\left( \left\vert x-a\right\vert 
            +\left\vert y-a\right\vert \right) 
  \leq \max\left\{\left\vert x-a\right\vert,\left\vert y-a\right\vert\right\}.
  \label{midpointineq}
\end{equation}%
Then $\max \left\{ \left\vert x-d\right\vert ,\left\vert y-d\right\vert
\right\} \leq \max \left\{ \left\vert x-a\right\vert ,\left\vert
y-a\right\vert \right\} $.

If $1\leq p<\infty $, inequality (\ref{convexcomb}) and the convexity of the
function $t\mapsto t^{p}$ on $\left( 0,\infty \right) $ imply $\left\vert
z-d\right\vert ^{p}\leq \left( 1-\lambda \right) \left\vert z-a\right\vert
^{p}+\lambda \left\vert z-m\right\vert ^{p}$. 
Adding the inequalities for $z=x$ and $z=y$ we obtain
\begin{equation*}
  \left\vert x-d\right\vert^{p}+\left\vert y-d\right\vert^{p}
  \leq \left(1-\lambda \right) 
     \left( \left\vert x-a\right\vert^{p}+\left\vert y-a\right\vert ^{p}\right) 
     +\lambda \left( \left\vert x-m\right\vert^{p}+\left\vert y-m\right\vert^{p}
    \right) .
\end{equation*}%
Again by convexity, inequality (\ref{midpointineq}) implies $\left\vert
x-m\right\vert ^{p}+\left\vert y-m\right\vert ^{p}\leq \left\vert
x-a\right\vert ^{p}+\left\vert y-a\right\vert ^{p}$. 
The latter two inequalities yield 
$\left\vert x-d\right\vert ^{p}+\left\vert y-d\right\vert^{p}
\leq \left\vert x-a\right\vert ^{p}+\left\vert y-a\right\vert ^{p}$.
\end{proof}

\begin{rem}
  In the case $p=1$ we do not need to assume that $G_{1}$ is
  midpoint convex. Let 
  $c$ be
  a point belonging to the intersection 
  $\left[ x,a\right] \cap \partial G_{1}$. 
  Then $\left\vert x-a\right\vert =\left\vert x-c\right\vert 
  +\left\vert c-a\right\vert $, 
  hence $\left\vert x-a\right\vert +\left\vert y-a\right\vert 
  \geq \left\vert x-c\right\vert+\left\vert y-c\right\vert $, 
  by the triangle inequality. 
  Then
  \[
     s_{G_{1}}(x,y)
     \geq \frac{\left\vert x-y\right\vert }
              {\left\vert x-c\right\vert +\left\vert y-c\right\vert }
     \geq \frac{\left\vert x-y\right\vert }
              {\left\vert x-a\right\vert +\left\vert y-a\right\vert }
     =s_{G}(x,y)\,.
  \]
\end{rem}

\begin{prop}\label{prop:triaineq}
   The Barrlund distance satisfies the triangle inequality.
\end{prop}

\begin{proof}
  The proof follows from a more general argument in
  \cite[Lemma 6.1]{h1}, but for the reader's convenience,
  we give a short argument here.
  Denote $b_p=b_{\mathbb{R}^n\setminus \{0\},p}.$ Let $x,y, z \in G\,.$
  Because $b_p$ is a metric by \cite{b}, for $u\in \partial G\,,$
  \begin{equation*}
    b_p(x-u,y-u)\le b_p(x-u,z-u)+b_p(z-u,y-u) \le b_{G,p}(x,z)+b_{G,p}(z,y) \,,
  \end{equation*}
  hence
  \begin{equation*}
    b_p(x-u,y-u)\le b_{G,p}(x,z)+b_{G,p}(z,y) \,.
  \end{equation*}
  Taking the supremum over $u\in \partial G,$ it follows that
  \begin{equation*}
    b_{G,p}(x,y)\le b_{G,p}(x,z)+b_{G,p}(z,y)\,.
   \qedhere
  \end{equation*}
\end{proof}

\begin{thm} \label{lem:monot}
  The Barrlund metric is monotone with respect to
  the parameter $p$: given a domain $ G\varsubsetneq \mathbb{R}^{n}$, 
  for $z_{1},~z_{2}\in G $ and $p>r \ge 1\,,$
  \begin{equation}  \label{eq:bsbd0}
    b_{ G ,r}(z_1,z_2)
    \le b_{ G ,p}(z_1,z_2)\le 2^{\frac1r - \frac1p }\,b_{ G ,r}(z_1,z_2)\,.
  \end{equation}
  In particular,
  \begin{equation}  \label{eq:bsbd}
    s_{ G }(z_1,z_2) \le b_{ G ,p}(z_1,z_2) \le 2^{1-1/p} s_{ G}(z_1,z_2)\,.
  \end{equation}
  Moreover, if $ n=2 $, then
  \begin{equation*}
    \sup \{ b_{ G,p}(z_1,z_2) : z_1,z_2 \in G \} = 2^{1-1/p} \,.
  \end{equation*}
\end{thm}

\begin{proof}
  The functions $p \mapsto ((a^p+ b^p)/2)^{1/p}$ and  $p \mapsto (a^p+ b^p)^{1/p}$
  are increasing and decreasing, respectively, on
  $(1, \infty)\,$
  for fixed $a,b>0\,.$  The monotonicity and \eqref{eq:bsbd0}
  follow from these basic facts and \eqref{eq:bsbd} is the special case
  $r=1$ of \eqref{eq:bsbd0}.
  For the proof of the last statement fix $x \in G$ and $z \in \partial G$
  with $d(x) = d(x,\partial G) =|x-z|\,$ and denote $w= (x+z)/2\,.$ Then for
  $\alpha \in (0, \pi/6)$ choose points $u_{\alpha} , v_{\alpha}$ with
  \begin{align*}
    & |u_{\alpha} -w| = |v_{\alpha}-w|= d(x)/2\,,
      \quad
      |u_{\alpha}-v_{\alpha}| = 2 d(x) \sin {\alpha} \cos {\alpha} \,, \\
    & |x- u_{\alpha}|=|x- v_{\alpha}|= d(x) \cos \alpha\,,
     \quad
    |z-u_{\alpha}|=|z- v_{\alpha}|= d(x) \sin \alpha\,.
  \end{align*}
  Applying the definition \eqref{eq:bdef} to the triple
  $u_{\alpha}\,,v_{\alpha}\,, z$ we have
  \begin{equation*}
    b_{G,p}(u_{\alpha},v_{\alpha} )
      \ge 
      \frac{2 d(x) \sin {\alpha} \cos {\alpha}}
            {d(x) \sqrt[p]{ \sin^p {\alpha} + \sin^p {\alpha}}}
      = 2^{1-1/p} \cos {\alpha}
      \to 2^{1-1/p} \,,
  \end{equation*}
  when ${\alpha}\to 0\,.$ This convergence together with \eqref{eq:bsbd}
  proves the claim.
\end{proof}

\begin{rem} {\rm \label{rmk:maxbarval}\
\setlength{\leftmargini}{20pt} 
\begin{enumerate}
 \item
  The supremum in Theorem  \ref{lem:monot}
  is attained for some domains, as shown below. 

  Let $p\geq 1$. Let $G=\mathbb{D\setminus }\left\{ 0\right\} $, 
  $t\in (0,1)$ and $z_{1}=t$, $z_{2}=-t$. 
  For every $z\in \partial \mathbb{D}$, 
  $|z_{1}-z|^{p}+|z-z_{2}|^{p}\geq 2^{1-p}\left( |z_{1}-z|+|z-z_{2}|\right)^{p}
    \geq 2^{1-p}|z_{1}-z_{2}|^{p}$ and both inequalities hold as equalities
  for $z=0$, hence  $b_{\mathbb{D},p}(z_{1},z_{2})=2^{1-\frac{1}{p}}$. The
  same argument shows that this holds in a more general case: if $G$ is a
  proper subdomain of $\mathbb{R}^{n}$ and there exist $z_{1},z_{2}\in G$,
  $z_{0}\in \partial G$ such that $z_{0}=(z_{1}+z_{2})/2$, 
  then $b_{G,p}(z_{1},z_{2})=2^{1-\frac{1}{p}}$. It follows that 
  \begin{equation*}
    \sup \left\{ \underset{z_{1},z_{2}\in G}{\sup }b_{G,p}(z_{1},z_{2}):G%
    \varsubsetneq \mathbb{R}^{n}\text{ is a domain}\right\} =2^{1-\frac{1}{p}}.
  \end{equation*}
\item
 We will see below in Theorem \ref{trivBnd} that the second
 inequality in \eqref{eq:bsbd} holds as equality for all $p\ge 1$
 if $G= \mathbb{H}\,,$
 $z_1,z_2 \in \mathbb{H}$ with $ \Im(z_1) = \Im(z_2)\,.$
\end{enumerate}
 } 
\end{rem}

Several upper and lower bounds for $s_G$ are given in \cite{hvz}.
Using these bounds and Theorem  \ref{lem:monot} one could find bounds also for
the Barrlund metric.

\subsection{\bf The proof of Theorem \ref{sbThm}.}
The proof follows from Theorem \ref{lem:monot}.\hfill $\square$

\medskip

We will next study a few problems which lead us to a formula
for the Barrlund metric when the domain is either the disk or the
half-plane.

\medskip

\noindent\textbf{Problem A.} 
For given $z_1,z_2 \in \mathbb{D}\,,$ find the contact
points and the corresponding parameter value $c>0$ of ``power $p$
ellipses'' $\{ |z_1-u|^p + |z_2-u|^p =c^p \}$ and the unit circle.

\smallskip

This Problem A is closely related to the following Problems A'.

\smallskip

\noindent\textbf{Problem A'.} 
For $z_1,z_2\in\mathbb{D} $ and $p\geq 1 $, find the
points $u$ on the unit circle $\partial\mathbb{D} $ such that
$\sqrt[p]{|z_1-u|^p+|z_2-u|^p} $
is minimal.

\begin{lem}
   Any point $u $ in Problem A' is given as a solution of
  \begin{equation}  \label{eq:prop}
    \Big((z_1\overline{z_1}+1)u-\overline{z_1}u^2-z_1\Big)^{\frac{p}2-1}
    (\overline{z_1}u^2-z_1) +\Big((z_2\overline{z_2}+1)u-
     \overline{z_2}u^2-z_2\Big)^{\frac{p}2-1} (\overline{z_2}u^2-z_2)=0\,,
  \end{equation}
  where we consider the principal branch of the complex power function.
\end{lem}

\begin{proof}
  We need to find the point $u$ on $ \partial\mathbb{D} $ such that
  $ |z_1-u|^p+|z_2-u|^p $ is minimal.
  Let
  \[
    G(\theta)
       =\Big((z_1-e^{i\theta})(\overline{z_1}-e^{-i\theta})\Big)^{\frac{p}2}
        +\Big((z_2-e^{i\theta})(\overline{z_2}-e^{-i\theta})\Big)^{\frac{p}2}.
  \]
  We remark that $G $ is a real-valued periodic function that is
  differentiable on the real line. Therefore, $G(\theta) $
  attains a global minimum at one point, which has to be a critical
  point of $G $. 
  For $ G'(\theta)=0 $, setting $u=e^{i\theta} $, we obtain \eqref{eq:prop}.
\end{proof}

The above equation \eqref{eq:prop} is no longer an algebraic
equation for a general real number $p> 1\,.$

\medskip

Next we give a counterpart of the above lemma for the upper half space.

\begin{lem} \label{halfplanextrem}
  Let $z_{1},z_{2}\in \mathbb{H}$ and $p\geq 1$. 
  The function $ S_{p}:\mathbb{R}\rightarrow \mathbb{R}$ defined by
  $S_{p}(t)=\left\vert t-z_{1}\right\vert^{p}+\left\vert t-z_{2}\right\vert^{p}$
  has a unique minimum point $t_{0}$. 
  If $\Re(z_{1})=\Re(z_{2}) $, then $t_{0}=\Re(z_{1})=\Re(z_{2}) $,
  otherwise 
  $\min \left\{ \Re(z_{1}),\Re(z_{2})\right\} <t_{0}
   <\max \left\{ \Re(z_{1}) ,\Re(z_{2})\right\} $ 
  and $t=t_{0}$ is the unique real solution of the equation.
  \begin{equation}\label{equt0}
    \big( t-\Re(z_{1}) \big)\left\vert t-z_{1}\right\vert^{p-2}
    =\big(\Re(z_{2})-t\big)\left\vert t-z_{2}\right\vert^{p-2}.  
  \end{equation}
\end{lem}

\begin{proof} 
For every $t\in \mathbb{R}$ we have
\[
  \displaystyle
  S_{p}^{\prime}(t)=p\sum_{k=1}^{2}\big(t-\Re(z_{k})\big)
                   \left\vert t-z_{k}\right\vert^{p-2}
\]
and%
\[
  \displaystyle
   S_{p}^{\prime\prime}(t)
      =p\sum_{k=1}^{2}\left[\left\vert t-z_{k}\right\vert^{p-2}
        +\left( p-2\right) \big(t-\Re(z_{k})\big)^{2}
         \left\vert t-z_{k}\right\vert^{p-4}\right] .
\]
Since $S_{p}^{\prime \prime }(t)>0$ for every $t\in \mathbb{R}$, the
derivative $S_{p}^{\prime }$ is increasing on $\mathbb{R}$.
Then $S_{p}$ is strictly 
convex on $\mathbb{R}$, hence, as 
$\underset{t\rightarrow \pm\infty }{\lim }S_{p}(t)=+\infty $, 
it follows that $S_{p}$ has a unique minimum point 
\cite[Theorems 3.4.4 and 3.4.5]{np}.

Note that $a<\min \left\{  {\Re}(z_{1}) , {\Re}%
(z_{2}) \right\} $ implies $S_{p}^{\prime }(a)<0$, while $b>\max
\left\{  {\Re}(z_{1}) ,  {\Re}(z_{2})
\right\} $ implies $S_{p}^{\prime }(b)>0$. Then the derivative $%
S_{p}^{\prime }$ has a unique zero $t_{0}$, which is the unique minimum
point of $S_f$. It follows that
\begin{equation*}
  b_{\mathbb{H},p}\left( z_{1},z_{2}\right)
   =\frac{\left\vert z_{1}-z_{2}\right\vert }
         {\sqrt[p]{\left\vert t_{0}-z_{1}\right\vert^{p}
    +\left\vert t_{0}-z_{2}\right\vert ^{p}}}.
\end{equation*}

\noindent{\bf Case 1.}  ${\Re}(z_{1})={\Re}(z_{2})$

    The derivative $S_{p}^{\prime }(t)=\big( t-{\Re}(z_{1})\big)\big(
    |t-z_{1}|^{p-2}+|t-z_{2}|^{p-2}\big) $, 
    $ t\in \mathbb{R}$ has the unique zero $t_{0}={\Re}%
    (z_{1})={\Re}(z_{2})$. Then
    \begin{equation*}
      b_{\mathbb{H},p}\left(z_{1},z_{2}\right)
        =\frac{\left\vert {\Im}\left( z_{1}\right)
          -{\Im}\left( z_{2}\right) \right\vert }{\sqrt[p]{{\Im}%
        \left( z_{1}\right) ^{p}+{\Im}\left( z_{2}\right) ^{p}}}.
    \end{equation*}

\noindent{\bf Case 2.} ${\Re}(z_{1})\neq {\Re}(z_{2})$.

    In this case,
    \begin{equation*}
      \min \left\{  {\Re}(z_{1}) , {\Re}(z_{2}) \right\}
        <t_{0}<\max \left\{  {\Re}(z_{1}) , {\Re} (z_{2}) \right\} .
    \end{equation*}
    Here
    $t_{0}$ is the unique real solution of the equation
    \begin{equation}
      \left( t-{\Re}(z_{1}\right) )\left\vert t-z_{1}\right\vert ^{p-2}
      =\left({\Re}(z_{2}\right) -t)\left\vert t-z_{2}\right\vert ^{p-2}.
      \label{equt1}
    \end{equation}%
    In the following we will assume that ${\Re}(z_{1})<{\Re}(z_{2})$,
    the case ${\Re}(z_{2})<{\Re}(z_{1})$ being analogous. For every $%
    t\in \mathbb{R}$ there exists a unique $\lambda =\lambda (t)\in \mathbb{R}$
    such that $t=\left( 1-\lambda \right) {\Re}\left( z_{1}\right) +\lambda
    {\Re}\left( z_{2}\right) $, and ${\Re}(z_{1})<t<{\Re}(z_{2})$ if
    and only if $0<\lambda (t)<1$. Then $\lambda =\lambda _{0}:=\lambda (t_{0})$
    is the unique solution of the equation
    \begin{equation}
      \lambda \left\vert \lambda {\Re}\left( z_{2}-z_{1}\right)
       -i {\Im}(z_{1})\right\vert ^{p-2}
       =\left( 1-\lambda \right) \left\vert \left(1-\lambda \right)
        {\Re}\left( z_{2}-z_{1}\right) +i {\Im}(z_{2})\right\vert ^{p-2}.
      \label{eqlambda} \qedhere
    \end{equation}
\end{proof}

\begin{rem} {\rm
  For $p=2$ we have
  $S_{p}^{\prime }(t)=4t-2\Re\left(z_{1}+z_{2}\right) $
  hence $t_{0}=\frac{1}{2}{\Re}\left(z_{1}+z_{2}\right) $
  and we obtain an alternative proof of Theorem \ref{binH}.

  In the general case, we can use (\ref{eqlambda}) for numerical computation
  of $\lambda _{0}$.}
\end{rem}

\subsection{\bf Barrlund's metric for $ p=1 \,.$} \label{B-1}
\subsubsection{\bf The domain $ G=\mathbb{H} $}\label{B-1-H}

The upper half space $\{z \in \mathbb{C}: \mathrm{Im}(z) > 0 \}$
is denoted by
$\mathbb{H}\,.$
Recall  that the hyperbolic metric in $\mathbb{H}$ is defined by the
formula \cite[Thm 7.2.1, p. 130]{be}
\[
  \cosh \rho _{\mathbb{H}}(z_1,z_2 )
    = 1 + \frac{|z_1 -z_2|^2}{2 \mathrm{Im} (z_1)   \mathrm{Im} (z_2) }\,,
     \quad
     z_1, z_2 \in  \mathbb{H} \,.
\]
Equivalently \cite[Thm 7.2.1, p. 130]{be},
\[
  \tanh \left( \frac{\rho _{\mathbb{H}}(z_1,z_2 )}{2}\right) =
  \frac{|z_1-z_2|}{|z_1-\overline{z}_2|} \,.
\]

  In the case $p=1$, (\ref{equt0}) in Lemma \ref{halfplanextrem}
  is equivalent to
  \begin{equation*}
    \frac{{\Re}\left( t-z_{1}\right) }{\left\vert t-z_{1}\right\vert }
     =\frac{{\Re}\left( z_{2}-t\right) }{\left\vert z_{2}-t\right\vert }.
  \end{equation*}%
  Assume that ${\Re}(z_{1})<{\Re}(z_{2})$. The above equality 
  holds for $t=(1-\lambda ){\Re}\left( z_{1}\right) +\lambda {\Re}(z_{2})$, 
  $ \lambda \in (0,1)$, if and only if the triangles $\Delta (z_{1},t,{\Re}%
  (z_{1}))$ and $\Delta (z_{2},t,{\Re}(z_{2}))$ are similar, that is, if
  and only if
  \begin{equation*}
   \frac{{\Re}\left( t-z_{1}\right) }{{\Re}\left( z_{2}-t\right) }
    =\frac{{\Im}(z_{1})}{{\Im}(z_{2})}
    =\frac{\left\vert t-z_{1}\right\vert }{\left\vert z_{2}-t\right\vert }
    =\frac{\lambda }{1-\lambda }.
  \end{equation*}%
  For $p=1$ we get $\lambda _{0}={{\Im}(z_{1})}/({{\Im}(z_{1})+%
  {\Im}(z_{2})})$, hence
  \[
    \left\vert t_{0}-z_{1}\right\vert
      =\lambda_{0}\left\vert z_{1}-\overline{z_{2}}\right\vert \,\,
         \quad{\rm and}\,\, \quad
    \left\vert t_{0}-z_{2}\right\vert
      = \left( 1-\lambda _{0}\right)
         \left\vert z_{1}-\overline{z_{2}}\right\vert \,
  \]
  hence we recover the formula
  \[
     s_{\mathbb{H}}(z_{1},z_{2})
       =b_{\mathbb{H},1}(z_{1},z_{2})
       =\frac{\left\vert z_{1}-z_{2}\right\vert }
        {\left\vert z_{1}-\overline{z_{2}}\right\vert }\,.
  \]

\subsubsection{\bf The domain $ G=\mathbb{D} $}\label{B-1-D}

\begin{rem} \label{eq:peq1}
  Substituting $p=1 $ into \eqref{eq:prop} and
  canceling the denominators, we have
  \begin{equation*}
    (\overline{z_1}u^2-z_1)
     \sqrt{(z_2\overline{z_2}+1)u-\overline{z_2}u^2-z_2}
    =- (\overline{z_2}u^2-z_2)  \sqrt{(z_1\overline{z_1}+1)u
     -\overline{z_1}u^2-z_1}\,.
  \end{equation*}
  Squaring the both sides and factorizing, 
  we have
  \[
    F \cdot\Big((\overline{z_1}-\overline{z_2})u^2
     -(\overline{z_1}z_2-z_1\overline{z_2})u  +z_2-z_1\Big)=0\,.
  \]
  The factor $ F $ coincides with the left hand side of the
  quartic equation \eqref{eq:equation}, and
  one of the roots gives the minimum.
\end{rem}

\subsection{\bf Barrlund's metric for $ p=2 \,.$} \label{B-2}

  The power $2$ ellipse is a circle. In fact, an equation of a
  power 2 ellipse $|z_1-w|^2+|z_2-w|^2=r^2 $, $ r>\frac{|z_1-z_2|}{2} $
  is expressed as
  $
    |2w-(z_1+z_2)|=\sqrt{2r^2-|z_1-z_2|^2}\,
  $.

\subsubsection{\bf The domain $ G=\mathbb{H} $}\label{B-2-H}

\begin{thm} \label{binH}
  For $z_1,z_2 \in \mathbb{H}$ we have
  \begin{equation*}
    b_{\mathbb{H}, 2}(z_1,z_2)
     = \frac{\sqrt{2} |z_1-z_2|}{\sqrt{ |z_1-z_2|^2 + |\mathrm{Im}(z_1+z_2)|^2 }}
     = \frac{ |z_1-z_2|}{\sqrt{ |z_1-m|^2 + |z_2-m|^2 }} \,,
  \end{equation*}
  where $m= \mathrm{Re}(z_1+z_2)/2 \,.$
\end{thm}

\begin{proof}
  Fix $z_1,z_2 \in \mathbb{H}\,$ and write $z=(z_1+z_2)/2\,.$ We
  will find
  \begin{equation*}
    \min \{ ( |z_1-u|^2 + |z_2-u|^2) : u \in \partial \mathbb{H} \}.
  \end{equation*}
By Remark \ref{CounterEx} (1),
  \begin{equation*}
    |u-z_1|^2 +|u-z_2|^2 = 2|u-z|^2 + \frac{1}{2}|z_1-z_2|^2 \,.
  \end{equation*}
  Then $|u-z_1|^2 +|u-z_2|^2$ attains its minimum if and only if $|u-z|$ does,
  i.e.
  if and only if $ u= m= \mathrm{Re}(z_1+z_2)/2\,.$ In conclusion,
  \begin{equation*}
    \min \{ ( |z_1-u|^2 + |z_2-u|^2) : u \in \partial \mathbb{H} \}
       =\frac{1}{2}(|z_1-z_2|^2 + |\mathrm{Im}(z_1+z_2)|^2)\,
  \end{equation*}
  and the desired formula follows.
\end{proof}

\begin{rem} {\rm
  By the definition of $s_{\mathbb{H}}\,,$ for $z_1,z_2 \in \mathbb{H}$
  \[
    s_{\mathbb{H}}(z_1,z_2 )= \frac{|z_1-z_2|}{|z_1-z| 
        + |z_2-z|}
      = \frac{|z_1-z_2|}{|z_1-\overline{z_2}| }
      = \tanh \frac{ \rho_{\mathbb{H}}( z_{1},z_{2})}{2 }\,
  \]
  where $\{z\} =[z_1, \overline{z_2}] \cap \mathbb{R} \,$ 
  \cite[Prop. 4.2]{HKLV}.

  We have by Theorem  \ref{lem:monot}
  \begin{equation*}
    s_{\mathbb{H}}(z_1,z_2 )\le b_{\mathbb{H},2}(z_1,z_2 )
       \leq \sqrt{2}s_{\mathbb{H}}(z_1,z_2 )=
       \sqrt{2}\tanh \dfrac{\rho _{\mathbb{H}}(z_1,z_2 )}{2} = 
      \sqrt{2}p_{\mathbb{H}}(z_1,z_2 ) 
  \end{equation*}%
  (see also \cite[Remark 6.2]{h1}).

  Moreover,
  $b_{\mathbb{H},2}\left( z_{1},z_{2}\right) =\sqrt{2}s_{\mathbb{H}}
    \left( z_{1},z_{2}\right) $
  if and only if $\mathrm{Im}(z_{1})=\mathrm{Im}(z_{2}).$
}
\end{rem}

\subsubsection{\bf The domain $ G=\mathbb{D} $}\label{B-2-D}

\begin{rem}
 \label{ex:barrpt}
  Substituting $p=2 $ into \eqref{eq:prop}, we have
  \begin{equation*}
    (\overline{z_1}u^2-z_1)+(\overline{z_2}u^2-z_2)
    =(\overline{z_1+z_2})u^2-(z_1+z_2)=0\,,
  \end{equation*}
  and
  $
    u=\pm\dfrac{z_1+z_2}{|z_1+z_2|}\, 
  $.
  Clearly, $u=\dfrac{z_1+z_2}{|z_1+z_2|} $ gives the minimum.
\end{rem}

\begin{thm} \label{cor:barrdist}
   For $z_1,z_2 \in \mathbb{D} $
  \begin{equation}  \label{eq:barrfor}
    b_{ \mathbb{D} ,2}(z_1,z_2)
      =\frac{|z_1-z_2|}{\sqrt{ 2+|z_1|^2 +|z_2|^2-2|z_1+z_2|}} \,.
  \end{equation}
  In particular, $\lim_{(0,1) \ni r \to 1}b_{ \mathbb{D} ,2}(r,t) =1$
  for $t \in (-1,1)\,.$
\end{thm}

\begin{proof} \ 
\begin{enumerate}
\setlength{\leftskip}{-20pt}
\setlength{\itemindent}{30pt}
  \item[{\bf Case 1.}] $z_{1}+z_{2}\neq 0\,.$

    \noindent
    Writing $u=(z_1+z_2)/|z_1+z_2|$ we see that
  $\overline{u}(z_1+z_2)= |z_1+z_2|$ and
  \begin{align*}
    |z_1-u|^2 + |z_2-u|^2
      &=2+ |z_1|^2 + |z_2|^2 - u ( \overline{ z_1 } +\overline{ z_2 })
         - \overline{ u }(z_1+z_2) \\
      &=2+|z_1|^2 +|z_2|^2 -2|z_1+z_2|\,.
  \end{align*}
  Applying Remark
  \ref{ex:barrpt} and substituting into
  \begin{equation*}
    b_{ \mathbb{D},2}(z_1,z_2) =\frac{|z_1-z_2|}{\sqrt{|z_1-u|^2 +|z_2-u|^2}}
  \end{equation*}
  yields the desired formula.
  \item[{\bf Case 2.}] $z_{1}+z_{2}=0$.

    \noindent
    For every $z\in \partial \mathbb{D}$, the segment joining $z$ to $0$ is a
    median in the triangle $\Delta \left( z,z_{1},z_{2}\right) $, therefore%
    \begin{equation*}
      \left\vert z-z_{1}\right\vert ^{2}+\left\vert z-z_{2}\right\vert ^{2}
        =2+\frac{1}{2}\left\vert z_{1}-z_{2}\right\vert ^{2}.
    \end{equation*}

    Then  $b_{\mathbb{D},\,2}(z_{1},z_{2})={\left\vert
    z_{1}-z_{2}\right\vert }/{\sqrt{2+\frac{1}{2}\left\vert
    z_{1}-z_{2}\right\vert ^{2}}}$, and
    \begin{equation*}
       \frac{1}{2}\left\vert z_{1}-z_{2}\right\vert ^{2}\Big|_{z_{2}=-z_{1}}
       =\left( \left\vert z_{1}\right\vert ^{2}+\left\vert z_{2}\right\vert ^{2}
        -2\left\vert z_{1}+z_{2}\right\vert \right) \Big|_{z_{2}=-z_{1}}
       =2\left\vert z_{1}\right\vert ^{2},
    \end{equation*}%
    therefore \eqref{eq:barrfor} holds. \qedhere
\end{enumerate}
\end{proof}

\begin{figure}[tbp]
  \centerline{\includegraphics[width=0.5\linewidth]{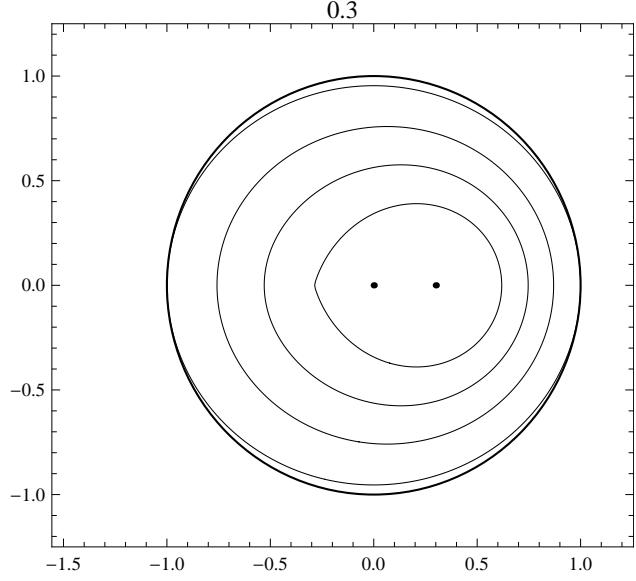}}
  \caption{Level sets $\{x+ i y: b_{{\mathbb{D}},2}(0.3, x+i y) =c \}$ for
           $c=0.4, 0.6, 0.8, 1.0\,$ and the unit circle.
           Note that for $c=1.0$ the level
           set meets the points $(\pm 1,0)\,$ in accordance with Theorem
           \ref{cor:barrdist}.}
  \label{fig:barrcircles}
\end{figure}

\medskip

Let $ B_{\mathbb{D},2}(a;c)=\{z\in\mathbb{D}\,:\, b_{\mathbb{D},2}(a,z)<c\} $.

\begin{thm} \label{thm322}
 Let $a$ and $ r $ be  numbers satisfying
$ b_{\mathbb{D},2}(a,a+r)=c $ and $ 0<a<a+r<1 \,.$ Then
\[
   \{ |z-a|<r\} \subset B_{\mathbb{D},2}(a;c) \subset \{|z|<a+r\}. 
\]
\end{thm}

\begin{proof}
We will prove that the inequalities
$
  b_{\mathbb{D},2}(a,a+re^{i\theta}) \leq b_{\mathbb{D},2}(a,a+r)
                               \leq b_{\mathbb{D},2}(a,(a+r)e^{i\theta})
$
hold for all $ \theta\in\mathbb{R} $.

Observe that $ b_{\mathbb{D},2}(w,z)=\frac{|w-z|}{\sqrt{2+|w|^2+|z|^2-2|w+z|}} $
holds for $ w,z\in\mathbb{D} $, by Theorem \ref{cor:barrdist}. 

At first, we will show
$  \big(b_{\mathbb{D},2}(a,a+r)\big)^2
   \leq \big(b_{\mathbb{D},2}(a,(a+r)e^{i\theta})\big)^2 $.
Let
\[
  u(\theta)
      =\big|a-(a+r)e^{i\theta}\big|^2\big(2+a^2+(a+r)^2-2(2a+r)\big)
         -r^2\big(2+a^2+(a+r)^2-2|a+(a+r)e^{i\theta}|\big).
\]
Then, $ u $ can also be written as
\begin{align*}
   u(\theta)=&2r^2\sqrt{2(ar+a^2)\cos\theta+(r^2+2ar+2a^2)} \\
            & +2((a-1)r^3+(3a^2-4a)r^2+(4a^3-6a^2+2a)r+2a^4-4a^3+2a^2)\\
            & -2a(r+a)((1-r-a)^2+(a-1)^2)\cos\theta.
\end{align*}
Set $ t=\cos\theta $ and $ u(\theta)=\tilde{u}(t) $.
Here we need to show $ \tilde{u}(t)\geq 0 $ holds for $ -1\leq t\leq 1 $.

The function $ \tilde{u}(t) $ has the unique critical point $ t_0 $
and attains the maximum at the point.
Moreover, we have
$\tilde{u}(1) =0 $ and
$ \tilde{u}(-1)=4a(1-r-a)\big(r(2-2a-r)+2a(1-a)\big) >0 $.
Therefore,
$  b_{\mathbb{D},2}(a,a+r)\leq b_{\mathbb{D},2}(a,(a+r)e^{i\theta}) $ holds
for for all $ \theta \in\mathbb{R} $.

The inequality
\[
    b_{\mathbb{D},2}(a,a+re^{i\theta})\leq b_{\mathbb{D},2}(a,a+r) \,,
\]
which holds by the proof  of Theorem \ref{BpThm}, completes the proof.
\end{proof}

It follows from \eqref{eq:hvz216} that the closures of
$s_{{\mathbb{D}} }$-disks centered at some point $z_0 \in {\mathbb{D}}$
are compact subsets of ${\mathbb{D}}\,.$
Looking at Figure \ref{fig:barrcircles}
we notice a topological difference:
 the $b_{{\mathbb{D}},2 }$-disks centered at
some point $(a,0)\,, a \in (-1,1)\,,$ with radius 
$1$ touch the boundary $%
\partial {\mathbb{D}}$ at the points $(\pm 1,0)\,.$
Moreover, it follows from \eqref{eq:barrfor} of Theorem \ref{cor:barrdist}
that $ b_{\mathbb{D},2} $-disk $ B_{\mathbb{D},2}(a;1) $ forms
the elliptic disk $ \big\{x+iy \,:\, x^2+\frac{y^2}{1-a^2}\leq 1\big\}$.

\begin{thm} \label{thm323}
 Let $a$ and $ r $ be  numbers satisfying $ b_{\mathbb{D},2}(a,a+r)=c $
and $ 0<a<a+r<1 $.
Then
\[
   B_{\mathbb{D},2}(a;c) \subset \{|z-a|<R\}\cap\mathbb{D},
\]
where $ R $ is the number satisfying
$ b_{\mathbb{D},2}(a,a-R)=c $ and $ -1<a-R<a $.
\end{thm}

\begin{proof}
We will show that
$
     b_{\mathbb{D},2}(a,a+r) \leq b_{\mathbb{D},2}(a,a-Re^{i\theta} )
$
holds for all $ \theta\in\mathbb{R} $.

As the value $ R $ satisfies $ b_{\mathbb{D},2}(a,a+r)=b_{\mathbb{D},2}(a,a-R) $,
the equality
\[
   \frac{r}{\sqrt{2+a^2+(a+r)^2-2(2a+r)}}
    =\frac{R}{\sqrt{2+a^2+(a-R)^2-2|2a-R|}} 
\]
follows from Theorem \ref{cor:barrdist}.
Squaring the both sides,
\begin{equation}\label{eq:conj-ineq}
  r^2\big(2+a^2+(a-R)^2-2|2a-R|\big)
  =R^2\big(2+a^2+(a+r)^2-2(2a+r)\big).
\end{equation}
Solving the equation (\ref{eq:conj-ineq}) for $ R $, we have
\renewcommand{\arraystretch}{1.5}
\[
    R=\left\{\begin{array}{ll}
          \frac{r(1-a)}{1-a-r} & \mbox{if} \ 2a-R\geq0
                       \ (\mbox{i.e.} \ 2a(1-a)-r(1+a)\geq0),\\
          \frac{r(1+a)}{1-a} & \mbox{if} \ 2a-R<0
                       \ (\mbox{i.e.} \ 2a(1-a)-r(1+a)<0).\\
             \end{array}  \right. 
\]
\renewcommand{\arraystretch}{1}
Here we consider the function
\[
  v(\theta)=|a+Re^{i\theta}|^2-2|2a+Re^{i\theta}|
   =a^2+2aR\cos\theta+R^2-2\sqrt{4a^2+4aR\cos\theta+R^2}.
\]
Set $ t=\cos \theta $  and $ v(\theta)=\widetilde{v}(t) $.
Then, $ \widetilde{v}(t) $ is convex downward in $ -1\leq t\leq 1 $ since
$ \widetilde{v} $ has the unique critical point and attains 
the minimum at the point.

At first, we will show that
$  \big(b_{\mathbb{D},2}(a,a+r)\big)^2
   \leq \big(b_{\mathbb{D},2}(a,a+Re^{i\theta})\big)^2 $
holds for $ R=\frac{r(1-a)}{1-a-r} $ and $ 2a-R>0 $.
Let $ \widetilde{u}_1(t)=R^2\big(2+a^2+(a+r)^2-2(2a+r)\big)
-r^2\big(2+a^2+\widetilde{v}(t)\big)$.
Then, $ \widetilde{u}_1 $ is concave in $ -1\leq t\leq 1 $,
and satisfies
$ \tilde{u}_1(1) = \frac{4r^3(1-a)^2}{1-r-a}  >0 $ and
$  \tilde{u}_1(-1) =0 $.
Therefore, $ \tilde{u}_1(t)\geq0 $ holds for $ -1\leq t\leq 1 $
and the assertion is obtained for this case.

Next, similarly, for $ R=\frac{r(1+a)}{1-a} $ and $ 2a-R<0 $, 
we have 
$ \tilde{u}_1(1) = \frac{4ar^2}{1-a}\big((1-a-r)(1+a)+(1-a)^2\big)>0 $
and
$  \tilde{u}_1(-1) = 0 $.
Therefore, $  \big(b_{\mathbb{D},2}(a,a+r)\big)^2
   \leq \big(b_{\mathbb{D},2}(a,a+Re^{i\theta})\big)^2 $
also holds for this case.

From the above arguments the assertion of the theorem is obtained.
\end{proof}  

\begin{rem}
The disk $ D(0,a+r)=\{|z|<a+r\} $ in Theorem \ref{thm322} always satisfies
$ D(0,a+r)\subset \mathbb{D} $,
but the disk $ D(a,R)=\{|z-a|<R\} $ in Theorem \ref{thm323}
may intersect  the unit circle.
So, there is no inclusion relation between these two disks
(see Figure \ref{pic:b2-bounds}).
\end{rem}

\begin{figure}[htbp]
   \centerline{\includegraphics[width=0.4\linewidth]{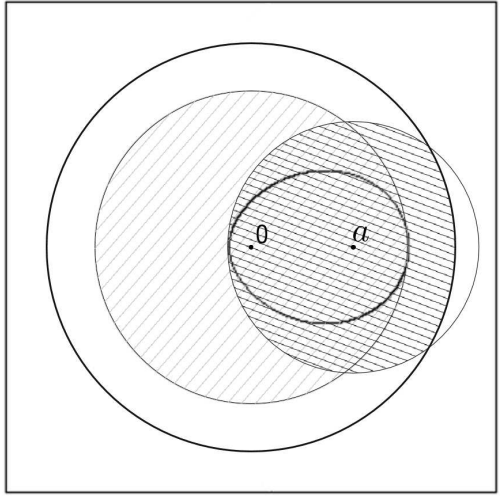}}
   \caption{The oval in the figure is the boundary of
            $B_{\mathbb{D},2}(a;0.5)$ with $a=0.5$.
            The disk with center the origin indicates
            the upper bound in Theorem \ref{thm322}.
            The shaded region corresponds to Theorem \ref{thm323}.}
   \label{pic:b2-bounds}
\end{figure}

\subsection{\bf Inequalities of Barrlund's metric for $ p \in (1, \infty)\,. $}
\label{B-p}

Let $ B_{\mathbb{D},p}(a;c)=\{z\in\mathbb{D}\,:\, b_{\mathbb{D},p}(a,z)<c\} $.

\medskip

\begin{thm} \label{BpThm}
The following holds for $ p>1 >a>0 $,
\[
     \{ |z-a|<r\} \subset B_{\mathbb{D},p}(a;c),
\]
where $ r $ is a number satisfying $ b_{\mathbb{D},p}(a,a+r)=c $
and $ 0<a<a+r<1 $.
\end{thm}

\begin{proof}
We will show the inequality
$
  b_{\mathbb{D},p}(a,a+re^{i\theta}) \leq b_{\mathbb{D},p}(a,a+r),
$
that is, we will show that
\begin{equation}\label{eq:conj-p2}
    \inf_{z\in\partial\mathbb{D}}\big(|a-z|^p+|a+r-z|^p\big)
     \leq\inf_{w\in\partial\mathbb{D}}\big(|a-w|^p+|a+re^{i\theta}-w|^p\big)
\end{equation}
holds for all $ \theta\in\mathbb{R} $.

The function $ |a-z|^p+|a+r-z|^p $ on the left hand side of \eqref{eq:conj-p2}
attains its minimum at $ z=1 $ because $ 0\leq a<a+r\leq 1 $.
Therefore, we see that
\begin{equation}\label{eq:conj-p3}
    \inf_{z\in\partial\mathbb{D}}\big(|a-z|^p+|a+r-z|^p\big)
                   = (1-a)^p+\big(1-(a+r)\big)^p .
\end{equation}
Since the distance between the point $ a+re^{i\theta} $ and the unit circle
is $ d_{\mathbb{D}}(a+re^{i\theta})=1-|a+re^{i\theta}| $, we have
\begin{align*}
   \inf_{w\in\partial\mathbb{D}}\big(|a-w|^p+|a+re^{i\theta}-w|^p\big)
      & \geq  \inf_{u\in\partial\mathbb{D}}|a-u|^p+
          \inf_{v\in\partial\mathbb{D}} |a+re^{i\theta}-v|^p\\
      &  =  (1-a)^p+ \big(1-|a+re^{i\theta}|\big)^p .
\end{align*}
Here, $ \big(1-|a+re^{i\theta}|\big)^p \geq \big(1-(a+r)\big)^p $ holds
as $ |a+re^{i\theta}|\leq a+r $ ($ \forall\theta\in\mathbb{R} $).
Hence, we have
\[
   \inf_{w\in\partial\mathbb{D}}\big(|a-w|^p+|a+re^{i\theta}-w|^p\big)
         \geq (1-a)^p+\big(1-(a+r)\big)^p
        =\inf_{z\in\partial\mathbb{D}}\big(|a-z|^p+|a+r-z|^p\big) ,
\]
and the assertion is obtained.
\end{proof}

\begin{lem} \label{mFX}
  For $ z_1,z_2\in\mathbb{D} \setminus \{0\}$, $z_1 \neq z_2$,
  and $p\ge1$ we have
  $\displaystyle  b_{\mathbb{D},p}(z_1,z_2)
         <b_{\mathbb{C} \setminus \overline{\mathbb{D}},p}
        \Big(\frac1{z_1},\frac1{z_2}\Big)$.
  In particular,
  $\displaystyle s_{\mathbb{D}}(z_1,z_2)
        <s_{\mathbb{C} \setminus \overline{\mathbb{D}}}
        \Big(\frac1{z_1},\frac1{z_2}\Big)$, 
  also holds (the case of $ p=1 $).
\end{lem}

\noindent
\begin{proof} 
At first, we observe that
\begin{align*}
   b_{\mathbb{C}\setminus \overline{\mathbb{D}},p}\Big(\frac1{z_1},\,\frac1{z_2}\Big)
     & =\sup_{w\in\partial\mathbb{D}}
       \frac{\big|\frac1{z_1}-\frac1{z_2}\big|}
            {\sqrt[p]{\big|\frac1{z_1}-w\big|^p+\big|w-\frac1{z_2}\big|^p}} \\
     & =\sup_{w\in\partial\mathbb{D}}
       \frac{|z_1-z_2|}{\sqrt[p]{|z_2|^p|1-wz_1|^p+|z_1|^p|1-wz_2|^p}}\,.
\end{align*}
Suppose that the functions
\[
   w \mapsto \sqrt[p]{|z_1-w|^p+|w-z_2|^p}
   \qquad \mbox{and}\qquad
   w \mapsto \sqrt[p]{|z_2|^p|1-wz_1|^p+|z_1|^p|1-wz_2|^p}
\]
defined on $\partial\mathbb{D}$ attain their minima at
$ u \in\partial\mathbb{D}$ and $ v\in\partial\mathbb{D}\, $, respectively.

Therefore,  we have
$ \displaystyle
  b_{\mathbb{D},p}(z_1,z_2)=\frac{|z_1-z_2|}{\sqrt[p]{|z_1-u|^p+|u-z_2|^p}}
$
and
\[
   b_{\mathbb{C}\setminus \overline{\mathbb{D}},p}\Big(\frac1{z_1},\,\frac1{z_2}\Big)
   =\frac{|z_1-z_2|}{\sqrt[p]{|z_2|^p|1-vz_1|^p+|z_1|^p|1-vz_2|^p}}\,.
\]
Then, for $ z_1,z_2\in\mathbb{D} $, we have
\begin{align*}
          |z_2|^p|1-vz_1|^p+|z_1|^p|1-vz_2|^p
     &  \leq
        |z_2|^p|1-\overline{u}z_1|^p+|z_1|^p|1-\overline{u}z_2|^p \\
     &  = |z_2|^p|u-z_1|^p+|z_1|^p|u-z_2|^p 
        < |u-z_1|^p+|u-z_2|^p \,.
\end{align*}
The first inequality holds from the assumption that the denominator
attains minima at $ v $,
and the second equality holds from $ u\overline{u}=1 $.
Hence,
\[
      \frac{|z_1-z_2|}{\sqrt[p]{|u-z_1|^p+|u-z_2|^p}}
      <\frac{|z_1-z_2|}{\sqrt[p]{|z_2|^p|1-vz_1|^p+|z_1|^p|1-vz_2|^p}}
\]
holds, and the assertion is obtained. 
\end{proof}

We give next a lower bound for $ b_{\mathbb{H}, p}, \ p\geq 1\,.$

\begin{thm} \label{trivBnd}
  For $z_1,z_2 \in \mathbb{H}$ and $p\ge 1$ let
  \[
    T_p(z_1, z_2)
      = \frac{ |z_1-z_2|}{|z_1-\overline{z}_2|
        \sqrt[p]{\alpha^p  + (1-\alpha)^p } }\,, \quad
    \alpha =  \frac{{\rm Im} (z_1)}{{\rm Im} (z_1) + {\rm Im} (z_2)}\,.
  \]
  Then
  \begin{equation} \label{tBineq}
    b_{\mathbb{H}, p}(z_1,z_2)
     \ge T_p(z_1,z_2) \ge \frac{ |z_1-z_2|}{|z_1-\overline{z}_2|}
     =  s_{\mathbb{H}}(z_1,z_2)\,      \,.
  \end{equation}
  In particular,
  $b_{\mathbb{H},\,1}(z_{1},z_{2})=T_{1}\left(z_{1},z_{2}\right)
    =s_{\mathbb{H}}(z_{1},z_{2}).$
  For $p>1$ the first inequality \eqref{tBineq} holds as an equality
  if and only if
  ${\Re}\left( z_{1}\right) ={\Re}\left( z_{2}\right) $
  or ${\Im}\left(z_{1}\right) ={\Im}\left( z_{2}\right) .$
\end{thm}

\begin{proof}
Fix $z_1,z_2 \in \mathbb{H}$ and let
$ \{w\} =[z_1,\overline{z}_2] \cap \mathbb{R}\,.$ By geometry
$ \displaystyle
 \frac{ |z_1-w|}{|z_1-\overline{z}_2|} = \alpha
$ 
and hence $|z_1-w| = \alpha |z_1-\overline{z}_2| \,.$ By the definition,
\[
  b_{\mathbb{H}, p}(z_1,z_2)
     \ge \frac{ |z_1-z_2|}{\sqrt[p]{|z_1-w|^p  
               +|z_2-w|^p }}= 
     \frac{ |z_1-z_2|}{|z_1-\overline{z}_2|\sqrt[p]{\alpha^p  + (1-\alpha)^p }}
      \,.
\]
Now we consider the equality cases.

Fix $p>1\,.$
The equality $ b_{\mathbb{H},\,p}(z_{1},z_{2})=T_{p}\left(z_{1},z_{2}\right) $
is equivalent to
\begin{equation*}
  \frac{\left\vert z_{1}-z_{2}\right\vert }
      {\sqrt[p]{\left\vert z_{1}-w\right\vert ^{p}
         +\left\vert z_{2}-w\right\vert ^{p}}}
  =\frac{\left\vert z_{1}-z_{2}\right\vert }
        {\underset{z\in \partial \mathbb{H}}\min
           \sqrt[p]{\left\vert z_{1}-z\right\vert ^{p}
                    +\left\vert z_{2}-z\right\vert ^{p}}}.
\end{equation*}
Assume that $z_{1}\neq z_{2}$. Then the above equality holds if and only if
\begin{equation} \label{inez}
  \left\vert z_{1}-z\right\vert ^{p}+\left\vert z_{2}-z\right\vert ^{p}
  \geq \left\vert z_{1}-w\right\vert ^{p}+\left\vert z_{2}-w\right\vert ^{p}
    \text{ for every }z\in \partial \mathbb{H}\,.
\end{equation}

\noindent\textit{\underline{Sufficiency}} \ 
By H\"{o}lder's inequality,
$
    \left\vert z_{1}-z\right\vert ^{p}+\left\vert z_{2}-z\right\vert ^{p}
      \geq 2^{1-p}\left( \left\vert z_{1}-z\right\vert
          +\left\vert z_{2}-z\right\vert \right) ^{p}\,.
$
By the definition of $w,$ we have
\[
   \underset{\zeta \in \partial \mathbb{H}}{\min }
     \left( \left\vert z_{1}-\zeta \right\vert +\left\vert z_{2}
        -\zeta\right\vert \right)
  =\left\vert z_{1}-w\right\vert +\left\vert z_{2}-w\right\vert \,,
\]
hence
\begin{equation*}
  \left\vert z_{1}-z\right\vert ^{p}+\left\vert z_{2}-z\right\vert ^{p}
   \geq 2^{1-p}\left( \left\vert z_{1}-w\right\vert +\left\vert z_{2}
     -w\right\vert\right) ^{p}\text{ for every }z\in \partial \mathbb{H} \,.
\end{equation*}

\noindent{\bf Case 1.}
    Assume that 
    ${\Im}\left( z_{1}\right) ={\Im}\left(z_{2}\right) $. 
    Then $\alpha =\frac{1}{2}$ and 
    $ \left\vert z_{1}-w\right\vert =\left\vert z_{2}-w\right\vert 
    =\frac{1}{2}\left\vert z_{1}-\overline{z_{2}}\right\vert \,$,
    therefore
    \begin{equation*}
      2^{1-p}\left( \left\vert z_{1}-w\right\vert +\left\vert z_{2}-w\right\vert
      \right) ^{p}=2\left\vert z_{1}-w\right\vert ^{p}=\left\vert
      z_{1}-w\right\vert ^{p}+\left\vert z_{2}-w\right\vert ^{p}.
    \end{equation*}
    It follows that \eqref{inez} holds.

\noindent{\bf Case 2.}
    Assume that 
    ${\Re}\left( z_{1}\right) ={\Re}\left(z_{2}\right) $. 
    Then $w={\Re}\left( z_{1}\right) ={\Re}\left(
    z_{2}\right) $. For every $z\in \partial \mathbb{H}$ we have
    \begin{equation*}
       \left\vert z_{k}-z\right\vert =\sqrt{{\Re}^{2}(z_{k}-z)+{\Im}%
       ^{2}\left( z_{k}\right) }\geq \left\vert {\Im}(z_{k})\right\vert
       =\left\vert z_{k}-w\right\vert
     \end{equation*}%
    for $k=1,2$, therefore,
    \eqref{inez} holds. 

\noindent\textit{\underline{Necessity}} \ 
Denote ${\Re}\left( z_{k}\right) =x_{k}$ 
for $k=1,2$. Then $w=\left( 1-\alpha \right)
x_{1}+\alpha x_{2}$.

Let $f(t)=\left\vert z_{1}-t\right\vert ^{p}+\left\vert z_{2}-t\right\vert
^{p}$, $t\in \mathbb{R}$. Since $t=w$ is a minimum point, it follows that $%
f^{\prime }\left( w\right) =0$.

But $f^{\prime }(t)=p\big( \left\vert z_{1}-t\right\vert ^{p-2}\left(
t-x_{1}\right) +\left\vert z_{2}-t\right\vert ^{p-2}\left( t-x_{2}\right)
\big) $, $t\in \mathbb{R}$. Then
\begin{eqnarray*}
  f^{\prime }(w)
   &=& p\big( \left\vert z_{1}-w\right\vert ^{p-2}\left(w-x_{1}\right)
       +\left\vert z_{2}-w\right\vert ^{p-2}\left( w-x_{2}\right)\big)  \\
   &=& p\left\vert z_{1}-\overline{z_{2}}\right\vert ^{p-2}
       \left(x_{2}-x_{1}\right) \big( \alpha ^{p-1}-\left( 1-\alpha \right)
      ^{p-1}\big) .
\end{eqnarray*}%
We see that $f^{\prime }\left( w\right) =0$ 
if and only if ${\Re}\left(
z_{1}\right) ={\Re}\left( z_{2}\right) $ or $\alpha =\frac{1}{2}$ 
(i.e. $ {\Im}\left( z_{1}\right) ={\Im}\left( z_{2}\right) $).
\end{proof}

\begin{rem}{\rm
  According to numerical tests, we have the following particular values
  \begin{align*}
    T_2(1+i6, -2+i 3) = 3/5\,, &\quad T_2(-4+i4, 4+i 12) = 4/5\, , \\
    T_p(-t+it, 1+i) =1 & \quad \mbox{ for all } p\ge 1\,, t>0\,.
  \end{align*}
}
\end{rem}
\begin{thm} \label{trivBnd2}
  For $z_1,z_2 \in \mathbb{H}$ and $p\ge 1$ let
  \begin{align*}
    & U_p(z_1, z_2) = \frac{ |z_1-z_2|}{\sqrt[p]{\alpha^p  + \beta^p } }\,,
     \quad \alpha =  \sqrt{{\rm Im} (z_1)^2+c^2}\,,
     \quad \beta =  \sqrt{{\rm Im} (z_2)^2+c^2}\,, \\
    & c=\left\vert {\rm Re}(z_{1}-z_{2})\right\vert /2.
  \end{align*}
Then
  \begin{equation} \label{tBineq2}
    b_{\mathbb{H}, p}(z_1,z_2)
     \ge U_p(z_1,z_2)       \,.
  \end{equation}
\end{thm}

\begin{proof}
Fix $z_1,z_2 \in \mathbb{H}$ and let  $u ={\Re}(z_1+z_2)/2
\,.$ The Pythagorean theorem yields
\[
  |z_1-u| = \alpha\,, |z_2-u|= \beta\,,
\]
and hence by the definition of the Barrlund metric the claim follows.
\end{proof}

We will compare below the above lower bounds $T_p$ and $U_p$ for
 the Barrlund metric. 

\begin{lem}
  For $z_{1},z_{2}\in \mathbb{H}$ let
  \begin{align*}
    & m=\frac{1}{2}({\Re}\left( z_{1}\right) +{\Re}\left( z_{2}\right) )\,,
    \quad
    \alpha =\frac{{\Im}(z_{1})}{{\Im}(z_{1})+{\Im}(z_{2})}\,,
   \quad
    w=(1-\alpha ){\Re}(z_{1})+\alpha {\Re}(z_{2})\,, \\
   & \quad  U_{p}\left( z_{1},z_{2}\right)
     :=\frac{\left\vert z_{1}-z_{2}\right\vert }
            {\sqrt[p]{\left\vert m-z_{1}\right\vert ^{p}+\left\vert m
           -z_{2}\right\vert^{p}}}\,\,,
         \quad T_{p}\left( z_{1},z_{2}\right)
     :=\frac{\left\vert z_{1}-z_{2}\right\vert }
            {\sqrt[p]{\left\vert w-z_{1}\right\vert^{p}
           +\left\vert w-z_{2}\right\vert ^{p}}}\,.
  \end{align*}
  If $p\geq 2$, then
  \[
     U_{p}\left(z_{1},z_{2}\right) \geq T_{p}\left( z_{1},z_{2}\right) \,.
  \]
\end{lem}

\begin{proof}
We will use Lemma \ref{halfplanextrem}.
Let $S_{p}(t)=\left\vert t-z_{1}\right\vert ^{p}+\left\vert
t-z_{2}\right\vert ^{p}$, $t\in \mathbb{R}$. We proved that the derivative $%
S_{p}^{\prime }$ is increasing on $\mathbb{R}$ and has a zero $t_{0}$, which
is the unique minimum point of $S_{p}$, since $S_{p}$ is decreasing on $%
(-\infty ,t_{0}]$ and increasing on $[t_{0},\infty ).$

With our notations,
\begin{equation}
  U_{p}\left( z_{1},z_{2}\right) -T_{p}\left( z_{1},z_{2}\right)
  =\frac{\left\vert z_{1}-z_{2}\right\vert }
        {\left( S_{p}(m)S_{p}(w)\right) ^{1/p}}%
  \left( \left( S_{p}(w)\right) ^{1/p}-\left( S_{p}(m)\right) ^{1/p}\right) .
  \label{inequt}
\end{equation}%
If $p=2$, we proved that $S_{p}(m)\leq S_{p}(t)$ for every $t\in \mathbb{R}$%
, in particular $S_{p}(m)\leq S_{p}(w)$, hence $U_{p}\left(
z_{1},z_{2}\right) \geq T_{p}\left( z_{1},z_{2}\right) $.

\medskip

Assume now that  $p>2$.

We have to compare $m$, $w$ and $t_{0}$.
\begin{equation*}
   m-w=\Big( \alpha -\frac{1}{2}\Big) {\Re}\left( z_{1}-z_{2}\right)
      =\frac{1}{2{\Im}(z_{1}+z_{2})}{\Re}\left( z_{1}-z_{2}\right) {\Im}
       \left( z_{1}-z_{2}\right) .
\end{equation*}

If ${\Im}(z_{1})={\Im}(z_{2})$ or ${\Re}(z_{1})={\Re}\left(
z_{2}\right) $, then $m=w$ and $U_{p}\left( z_{1},z_{2}\right) =T_{p}\left(
z_{1},z_{2}\right) $ for every $p\geq 2$ and the claim follows.

Now assume that ${\Re}(z_{1})\neq {\Re}\left( z_{2}\right) $ and $%
{\Im}(z_{1})\neq {\Im}(z_{2})$.

Let $g_{p}(\lambda )=S_{p}^{\prime }((1-\lambda ){\Re}\left(
z_{1}\right) +\lambda {\Re}\left( z_{2}\right) )$, $\lambda \in \left[
0,1\right] $. We have
\begin{align*}
  g_{p}(\lambda )= & p{\Re}\left( z_{2}-z_{1}\right) \\
                  & \times \Big[ \lambda \big\vert
            \lambda {\Re}\left( z_{2}-z_{1}\right) -i{\Im}(z_{1})
             \big\vert^{p-2}-\left( 1-\lambda \right)\big\vert
             \left( 1-\lambda \right) {\Re}%
             \left( z_{2}-z_{1}\right) +i{\Im}(z_{2})\big\vert ^{p-2}\Big] .
\end{align*}
Then
\begin{equation*}
  g_{p}\Big(\frac{1}{2}\Big)
    =\frac{p}{2}{\Re}\left(z_{2}-z_{1}\right)
      \bigg[ \Big\vert \frac{1}{2} {\Re}\left(z_{2}-z_{1}\right)
          -i{\Im}(z_{1})\Big\vert ^{p-2}-\Big\vert \frac{1}{2}{\Re}
    \left( z_{2}-z_{1}\right) +i {\Im}(z_{2})\Big\vert ^{p-2}\bigg] .
\end{equation*}

Then
\begin{equation*}
   {\Re}\left( z_{1}-z_{2}\right) {\Im}(z_{1}-z_{2})
   g_{p}\Big( \frac{1}{2}\Big) <0,
\end{equation*}%
since $p>2$.

\noindent{\bf Case 1.}
    ${\Re}\left( z_{1}-z_{2}\right) {\Im}\left(z_{1}-z_{2}\right) >0$.

   We have $w<m$. On the other hand, $g_{p}\left( \frac{1}{2}\right) <0$, hence
   $m<t_{0}$. Since $w<m<t_{0}$ and $S_{p}$ is decreasing on
   $(-\infty ,t_{0}]$, we have $S_{p}(w)\geq S_{p}(m)$.

\noindent{\bf Case 2.}
   ${\Re}\left( z_{1}-z_{2}\right) {\Im}\left(z_{1}-z_{2}\right) <0$.

   Now $w>m$ and $g_{p}\left( \frac{1}{2}\right) >0$, hence $m>t_{0}$. 
   Since $ w>m>t_{0}$ and $S_{p}$ is increasing on $[t_{0},\infty )$, 
   we have $ S_{p}(w)\geq S_{p}(m)$. In both cases
   inequality (\ref{inequt}) shows that 
   $ U_{p}\left( z_{1},z_{2}\right) -T_{p}\left( z_{1},z_{2}\right) \geq 0$. 
  \qedhere
\end{proof}

\subsection{\bf Barrlund's metric for $p=\infty\,.$}\label{B-inf}

Let $G\subset {\mathbb{R}}^{n}$ be a  proper subdomain. Let
\[
  b_{G,\infty }(z_1,z_2)=\sup_{w\in \partial G}\frac{|z_1-z_1|}
  {\max \left\{|z_1-w|,|z_2-w|\right\} }.
\]
For $G=\mathbb{R}^{n}\setminus \{0\}$, D. Day \cite{d}
proved that $b_{G,\infty }$ is a metric.

Note that $\max \left\{ |z_1-w|,|z_2-w|\right\}
=\underset{p\rightarrow \infty }{\lim }\sqrt[p]{|z_1-w|^{p}+|z_2-w|^{p}}$.
It follows that
\[
  b_{G,p}(z_1,z_2)\leq b_{G,\infty }(z_1,z_2)\leq 2^{\frac{1}{p}}b_{G,p}(z_1,z_2)
\]%
for all $z_1,z_2\in G$ and $1\leq p<\infty $.

\

Recall that the power $ p $ ellipse $ E_p $ is written as
$ |z-z_1|^p+|z-z_2|^p=r^p $.
We have the following result for the shape of the power $ \infty $ ellipse.

\begin{lem} \label{lem:pinfty}
The power $\infty$ ellipse is given by
\[
   E_{\infty}\,:\  \partial\{|z-z_1|<r \mbox{ and } |z-z_2|<r\}.
\]
\end{lem}

\begin{proof}

The assertion holds from
  $ \displaystyle \lim_{p\to\infty}\sqrt[p]{|z-z_1|^p+|z-z_2|^p}
    =\max\{|z-z_1|,|z-z_2|\} $.
\end{proof}

\subsubsection{\bf The domain $ G=\mathbb{H}$.}\label{B-inf-H}

\renewcommand{\arraystretch}{2}
\begin{thm}\label{binftyform}
  For $z_1,z_2 \in \mathbb{H}$
  \[
      b_{\mathbb{H},\infty}(z_1,z_2)= 
      \left\{
      \begin{array}{l}
         \dfrac{2|{\Re} (z_1-z_2)|}{|z_1-\overline{z_2}|}
             \quad\text{if } \min\{{\Re} (z_1),{\Re} (z_2)\} <\tilde{z}
               < \max\{{\Re} (z_1),{\Re} (z_2)\},\\
         \dfrac{|z_1-z_2|}{\max\{{\Im} (z_1),\ {\Im} (z_2)\}} \quad
             \mbox{otherwise}.
     \end{array}\right. 
  \]
\renewcommand{\arraystretch}{1}
  where $ \tilde{z}=\frac{\overline{z_1}z_1-\overline{z_2}{z_2}}
                {(z_1-z_2)+(\overline{z_1}-\overline{z_2})} $ if
                 $\Re (z_1)\neq \Re (z_2)\,.$
\end{thm}

\begin{proof}
Assume first that $\Re (z_1) \neq \Re (z_2)$.
Let $ \tilde{z} $ be the intersection point of the real axis
and the perpendicular bisector $ \ell $ of the segment $ [z_1,z_2] $.
The line $ \ell $ and $ \tilde{z} $ are given by
\[
   \ell : \ (\overline{z_1}-\overline{z_2})z+(z_1-z_2)\overline{z}
             =\overline{z_1}z_1-\overline{z_2}{z_2} \quad \mbox{and} \quad
   \tilde{z}  = \frac{\overline{z_1}z_1-\overline{z_2}{z_2}}
                      {(z_1-z_2)+(\overline{z_1}-\overline{z_2})}.
\]
Then, we need to consider the following two cases.
\begin{figure}[htbp]
  \centerline{\includegraphics[width=0.4\linewidth]{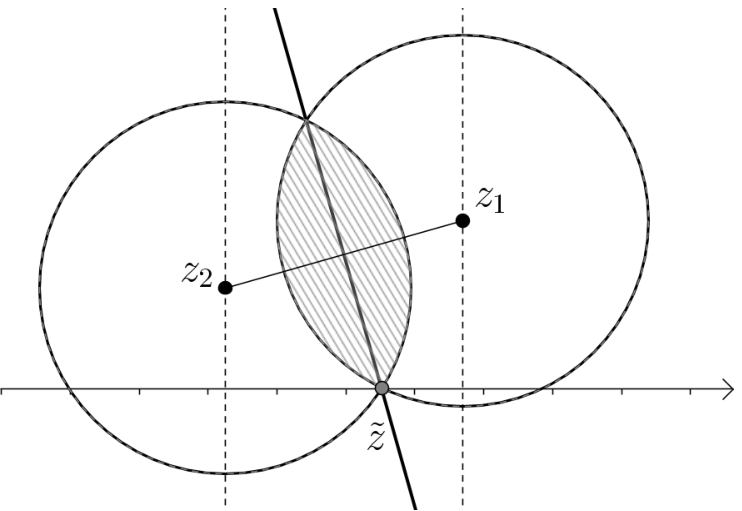}\qquad
              \includegraphics[width=0.4\linewidth]{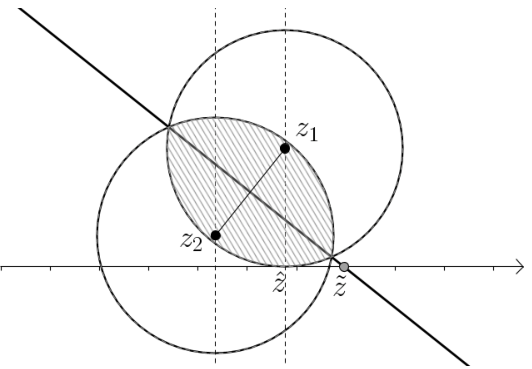}}
  \caption{The left and right figures indicate the case
            (\ref{proof:434:1}) and (\ref{proof:434:2}) respectively.}
  \label{fig:power-inf-ellips}
\end{figure}

\begin{enumerate}
  \item \label{proof:434:1}
        $ \min\{{\Re} (z_1),\ {\Re} (z_2)\} \leq \tilde{z}
           \leq  \max\{{\Re} (z_1),\ {\Re} (z_2)\} $

        The limit
        $ \displaystyle \lim_{p\to\infty}\sqrt[p]{|z_1-z|^p+|z-z_2|^p}
          =\max \left\{ \left\vert z_{1}-z\right\vert ,
           \left\vert z_{2}-z\right\vert\right\} $

        attains the minimum at $ z=\tilde{z} $ and its minimum is
        \[
            |z_1-\tilde{z}|
            =\Big|\frac{(z_1-z_2)(z_1-\overline{z_2})}{2{\Re}(z_1-z_2)}
             \Big|.
        \]
        Therefore in this case,
        \[
            b_{\mathbb{H},\infty}(z_1,z_2)=\frac{2|{\Re} (z_1-z_2)|}
                                            {|z_1-\overline{z_2}|}.
        \]
  \item \label{proof:434:2}
        $ \tilde{z}\leq \min\{{\Re} (z_1),\, {\Re} (z_2)\} $ or
        $ \max\{{\Re} (z_1),\, {\Re} (z_2)\}\leq \tilde{z} $

        In this case,
        \[
           \max \left\{ \left\vert z_{1}-z\right\vert ,
             \left\vert z_{2}-z\right\vert\right\}
        \]
        attains the minimum at the finite endpoint of
        the interval 
        (for example, point $\hat{z}$ on the Figure \ref{fig:power-inf-ellips})
        where $ \tilde{z} $ belongs
        and the minimum is   ${\max\{{\Im} (z_1),\ {\Im} (z_2)\}}$.
  Then
        \[
               b_{\mathbb{H},\infty}(z_1,z_2)=\frac{|z_1-z_2|}
                        {\max\{{\Im} (z_1),\ {\Im} (z_2)\}}.
        \]
  If $\Re (z_1) =\Re (z_2)\,,$ then the above formula also holds.
\qedhere
\end{enumerate}
\end{proof}

An upper bound of $ b_{\mathbb{H},p}(z_1,z_2) $ is given as follows.

\begin{prop} For $z_1,z_2 \in \mathbb{H}$
\[
    b_{\mathbb{H},p}(z_1,z_2)\leq
         \frac{|z_1-z_2|}{\max\{{\Im}(z_1), {\Im} (z_2)\}}.
\]
\end{prop}

\begin{proof}
From Theorem \ref{lem:monot}, \eqref{eq:bsbd0}, the inequality
\[
   s_{\mathbb{H}}(z_1,z_2)\leq b_{\mathbb{H},p}(z_1,z_2)
                        \leq b_{\mathbb{H},\infty}(z_1,z_2)
\]
holds.
Also, from the proof of the above lemma the inequality
\[
   |\max\{{\Im}(z_1),{\Im}(z_2)\}| \leq | z_k-\tilde{z} |
\]
($ k=1,2 $) holds. Therefore, we have
\[
   \frac{2|{\Re} (z_1-z_2)|}{|z_1-\overline{z_2}|} \leq
    \frac{|z_1-z_2|}{\max\{{\Im} (z_1), {\Im} (z_2)\}},
\]
and the assertion is obtained.
\end{proof}

\subsubsection{\bf The domain $ G=\mathbb{D}\,$}\label{B-inf-D}

\begin{lem}
Suppose $ z_1,z_2\in\mathbb{D} $ satisfy $ r=|z_1| \leq |z_2| \,.$
Set $ z_1=re^{i\theta} \,.$

Then, the following (\ref{lem-item:1}), (\ref{lem-item:2}) 
and (\ref{lem-item:3}) are equivalent to each other.
\begin{enumerate}
  \item \label{lem-item:1}
         $ b_{\mathbb{D},\infty}(z_1,z_2) $ attains its supremum
    at $ u=\frac{z_1}{|z_1|}=e^{i\theta}\,.$
  \item \label{lem-item:2}
        $  z_2\in\big\{|z-e^{i\theta}|\leq 1-r\big\}\cap \mathbb{D} \,.$
  \item \label{lem-item:3} 
        the power $\infty$ ellipse
     $ \lim_{p\to\infty}\sqrt[p]{|z-z_1|^p+|z-z_2|^p}=1-r $
    tangents to the unit circle.
\end{enumerate}
\end{lem}

\begin{proof}
$ (\ref{lem-item:1})\Leftrightarrow (\ref{lem-item:3}) $ \
The power $\infty$ ellipse in (3) is written as 
\[
   \partial\{ |z-z_1|\leq 1-r \mbox{ and } |z-z_2|\leq 1-r\}.
\]
The circle $ |z-z_1|=1-r $ is inscribed
in the unit circle, and the point $ \frac{z_1}{|z_1|}=e^{i\theta} $ is the
point of tangency of these two circles.
In this case, if power $\infty $ ellipse with foci $ z_1 $ and $ z_2 $
tangent to the unit circle at a point in its ``arc'',
the point of tangency is also given by $ u=e^{i\theta} $
(see the left figure in Figure \ref{pic:1}).
Clearly, the converse also holds.

\begin{figure}[htbp]
\centerline{
\includegraphics[width=0.4\linewidth]{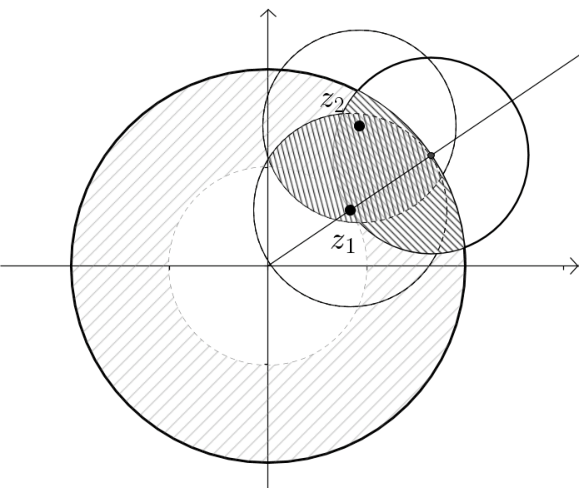}\qquad
\includegraphics[width=0.4\linewidth]{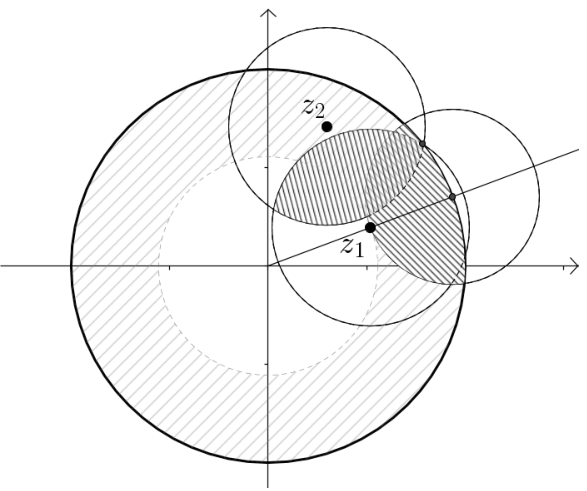}}
\caption{The power $\infty $ ellipse and the set
  $ \big\{|z-e^{i\theta}|\leq 1-r\big\}\cap \mathbb{D} $.}
\label{pic:1}
\end{figure}

\smallskip

\noindent
$ (\ref{lem-item:1}) \Rightarrow (\ref{lem-item:2}) $ \
From the above argument, the following is also obtained:
if the unit circle tangent to
a power $\infty $ ellipse at a point in ``arc'',
$ b_{\mathbb{D},\infty}(z_1,z_2) $ attains its supremum
at the tangent point $ u=\frac{z_1}{|z_1|} \,.$

Here we consider the case when the unit circle
intersects with a power $\infty$ ellipse at one of the vertices.
Let $ D $ be the set consisting of the points $ z_2 $ 
in which
$ b_{\mathbb{D},\infty} $ attains its supremum at a vertex of
corresponding power $\infty $ ellipse.
Then, for each boundary point $ z_2 \in\partial D\,, $
$ b_{\mathbb{D},\infty}(z_1,z_2) $ attains the supremum at the vertex
$ u=e^{i\varphi} $ of power $ \infty $ ellipse.

Now, let $ \ell $ be the line passing through
$ e^{i\theta} $ and $ e^{i\varphi}\,, $
and $ z^* $ the reflection point of $ z_1 $ with respect to the line
$ \ell \,.$
Then, we have
\[
     \ell\ :\ z+e^{i\theta}e^{i\varphi}\overline{z}=e^{i\theta}+e^{i\varphi}\,,
     \quad\mbox{and}\quad
     z^*=e^{i\theta}+(1-r)e^{i\varphi} \,.
\]
The trace of $ z^* $ forms the circle
\begin{equation}\label{eq:infty-circle}
     |z-e^{i\theta}|=1-r\,,
\end{equation}
as the point $ e^{i\varphi} $ ranges over the unit circle.
Clearly, if we choose the point $ z_2 $ in the inside of the disk
(\ref{eq:infty-circle}), the unit circle
tangents to a power $\infty$ ellipse with tangency a point in ``arc''.

\smallskip

\noindent
$ (\ref{lem-item:2})\Rightarrow (\ref{lem-item:3}) $
From the above argument, it is clear that
if $ z_2 $ is in the disk $ |z-e^{i\theta}|\leq 1-r $
(and $ z_2\in\mathbb{D} $),
the power $\infty$ ellipse with foci $ z_1,z_2 $
is inscribed in the unit circle and the tangent point is a point in ``arc''
part of the power $\infty$ ellipse.
As the distance from $ z_1 $ to the unit circle is $ 1-r \,,$
the power $\infty$ ellipse is written by
$ \lim_{p\to\infty}\sqrt[p]{|z-z_1|^p+|z-z_2|^p}=1-r \,.$
\end{proof}

\begin{thm}
  Let $z_{1},z_{2}\in \mathbb{D\setminus }\left\{ 0\right\} $ be distinct
  points. Then 
\[
     b_{\mathbb{D},\,\infty }(z_{1},z_{2})= \left\{
      \begin{array}{l}
       \dfrac{|z_{1}-z_{2}|}{1-\min \big\{|z_{1}|,|z_{2}| \big\}}\quad
     \text{if }  |z_{1}| \leq 1-\big|z_{2}-\frac{z_{1}}{|z_{1}|}\big|
      \text{ or }  |z_{2}| \leq 1-\big|z_{1}-\frac{z_{2}}{|z_{2}|}\big|,\\
      \dfrac{|z_{1}-z_{2}|}{\min \big\{|z'-z_{1}|,|z''-z_{1}|\big\}}\quad
      \text{otherwise}.
   \end{array}%
   \right. 
\] 
Here $z^{\prime }$ and $z^{\prime \prime }$ are the intersections of the
perpendicular bisector of the segment $[z_{1},z_{2}]$ with the the unit
circle $\partial \mathbb{D}$, and are given by
\begin{equation}
  \left\{ z^{\prime },z^{\prime \prime }\right\}
  =\left\{ \frac{z_{1}-z_{2}}{\left\vert z_{1}-z_{2}\right\vert }
      \Bigg( \frac{\left\vert z_{1}\right\vert^{2}
     -\left\vert z_{2}\right\vert ^{2}}{2\left\vert z_{1}-z_{2}\right\vert }%
     \pm i\sqrt{1-\bigg( \frac{\left\vert z_{1}\right\vert ^{2}-\left\vert
    z_{2}\right\vert ^{2}}{2\left\vert z_{1}-z_{2}\right\vert }\bigg) ^{2}}%
    \, \Bigg) \right\} .  \label{intersecmed}
\end{equation}
\end{thm}

\begin{proof}
Let $z_{1},z_{2}\in \mathbb{D}$. Denote $M(z):=\max \left\{ \left\vert
z-z_{1}\right\vert ,\left\vert z-z_{2}\right\vert \right\} $,
$z\in \mathbb{C}$ and $m:=\underset{z\in \partial \mathbb{D}}{\min }M(z)$.
Then
\[
  b_{\mathbb{D},\,\infty }(z_{1},z_{2})=\frac{\left\vert
  z_{1}-z_{2}\right\vert }{m}.
\]

If $z_{1}=z_{2}$, then $m=1-\left\vert z_{1}\right\vert $ and
$b_{\mathbb{D},\,\infty }(z_{1},z_{2})=0$.
If $z_{1}=0\neq z_{2}$ or $z_{2}=0\neq z_{1}$, then $m=1$.
In the following we assume  that
$z_{1},z_{2}\in \mathbb{D\setminus }\left\{ 0\right\} $ are distinct.

The perpendicular bisector $\mathcal{L}$ of the segment $[z_{1},z_{2}]$ has
the equation $\mathcal{L}:$ $L(z)=0$, where
\begin{equation*}
  L(z)=\left( \overline{z_{1}}-\overline{z_{2}}\right) z+\left(
       z_{1}-z_{2}\right) \overline{z}-\left( \left\vert z_{1}\right\vert^{2}
       -\left\vert z_{2}\right\vert ^{2}\right) .
\end{equation*}%
The closed half-planes determined by $\mathcal{L}$ are
$H_{1}=\left\{ z\in \mathbb{C}:L(z)\geq 0\right\} $ and
$H_{2}=\left\{ z\in \mathbb{C}:L(z)\leq 0\right\} $.
Since $L(z_{1})=\left\vert z_{1}-z_{2}\right\vert ^{2}>0$ and
$ L(z_{2})=-L(z_{1})<0$, we have $z_{k}\in H_{k}\setminus \mathcal{L}$ for
$ k=1,2$.
Note that
$L(0)=\left\vert z_{2}\right\vert ^{2}-\left\vert z_{1}\right\vert ^{2}$ and
\renewcommand{\arraystretch}{1.2}
\begin{equation*}
  M(z)=\left\{
              \begin{array}{l}
                 \left\vert z-z_{2}\right\vert \text{ \ if }z\in H_{1},\\
                 \left\vert z-z_{1}\right\vert \text{ \ if }z\in H_{2}.
              \end{array}%
\right. 
\end{equation*}

Then $m=\min \left\{ m_{1},m_{2}\right\} $, where
$ \displaystyle m_{1}:=\min_{z\in\partial \mathbb{D\cap }H_{2}}
  \left\vert z-z_{1}\right\vert $ and
$ \displaystyle m_{2}:=\min_{z\in \partial \mathbb{D\cap }H_{1}}
\left\vert z-z_{2}\right\vert $.

The minimum in the definition of $m_{1}$ is attained at
$z=\frac{z_{1}}{\left\vert z_{1}\right\vert }$ if
$\frac{z_{1}}{\left\vert z_{1}\right\vert }\in H_{2}$,
respectively at some $z\in \left\{ z^{\prime },z^{\prime \prime}\right\} $
if $\frac{z_{1}}{\left\vert z_{1}\right\vert }\in H_{1}$.
Then $m_{1}=1-\left\vert z_{1}\right\vert $ if
$\frac{z_{1}}{\left\vert z_{1}\right\vert }\in H_{2}$ and
$m_{1}=\min \left\{ \left\vert z^{\prime}-z_{1}\right\vert ,
\left\vert z^{\prime \prime }-z_{1}\right\vert \right\} $
if $\frac{z_{1}}{\left\vert z_{1}\right\vert }\in H_{1}$.

Denote $m_{3}:=1-\min \left\{ \left\vert z_{1}\right\vert ,
\left\vert z_{2}\right\vert \right\} $ and
$m_{4}:=\min \left\{ \left\vert z^{\prime}-z_{1}\right\vert ,
\left\vert z^{\prime \prime }-z_{1}\right\vert \right\}
=\min \left\{ \left\vert z^{\prime }-z_{2}\right\vert ,
\left\vert z^{\prime\prime }-z_{2}\right\vert \right\} $.
Note that $m_{4}\geq m_{3}.$

We will assume that $\left\vert z_{1}\right\vert \leq \left\vert
z_{2}\right\vert $, equivalently, $0\in H_{1}$. The case $\left\vert
z_{2}\right\vert \leq \left\vert z_{1}\right\vert $ is similar.

$0\in H_{1}$ yields $\frac{z_{2}}{\left\vert z_{2}\right\vert }\in H_{2}$,
otherwise by the convexity of $H_{1}$ we get $z_{2}\in H_{1}$, which is
false. So, $0\in H_{1}$ implies $m_{2}=m_{4}$.

If $0\in H_{1}$ and $\frac{z_{1}}{\left\vert z_{1}\right\vert }\in H_{1}$,
then  $m_{1}=m_{4}$, hence $m=m_{4}$.
If $0\in H_{1}$ and $\frac{z_{1}}{\left\vert z_{1}\right\vert }\in H_{2}$,
then $m_{1}=1-\left\vert z_{1}\right\vert =m_{3}\leq m_{4}$,
hence $m=m_{3}$.

\smallskip

We obtain
\begin{equation*}
  m=\left\{
       \begin{array}{l}
         m_{4}\text{ \ if }
           (0\in H_{1}\text{ and }\frac{z_{1}}{|z_{1}|}\in H_{1})\text{ or }
           (0\in H_{2}\text{ and }\frac{z_{2}}{|z_{2}|}\in H_{2}),\\
        m_{3}\text{ \ if }
           (0\in H_{1}\text{ and }\frac{z_{1}}{|z_{1}|}\in H_{2})\text{ or }
           (0\in H_{2}\text{ and }\frac{z_{2}}{|z_{2}|}\in H_{1}).
\end{array}%
\renewcommand{\arraystretch}{1}
\right.
\end{equation*}

In particular, there are the following special cases.
If $0\in H_{1}\cap H_{2}$
(i.e. $\left\vert z_{1}\right\vert =\left\vert z_{2}\right\vert $),
then $\frac{z_{1}}{\left\vert z_{1}\right\vert }\in H_{1}$ and
$\frac{z_{2}}{\left\vert z_{2}\right\vert }\in H_{2}$,
hence $m=m_{4}$. If $\frac{z_{1}}{\left\vert z_{1}\right\vert },
\frac{z_{2}}{\left\vert z_{2}\right\vert }\in H_{1}\cap H_{2}$,
then $m=m_{3}=m_{4}$.

Since $L\big(\frac{z_{1}}{\left\vert z_{1}\right\vert }\big)
=\big\vert z_{2}-\frac{z_{1}}{\left\vert z_{1}\right\vert }\big\vert ^{2}
-\left(1-\left\vert z_{1}\right\vert \right) ^{2}$,
we have $\frac{z_{1}}{\left\vert z_{1}\right\vert }\in H_{2}$ if and only if
\begin{equation*}
  E(z_{1},z_{2}):=\left\vert z_{2}-\frac{z_{1}}{\left\vert z_{1}\right\vert }%
  \right\vert -\left( 1-\left\vert z_{1}\right\vert \right) \leq 0,
\end{equation*}%
i.e. $z_{2}$ belongs to the closed disk bounded by the circle
$\mathcal{C}_{1}$ centered at $\frac{z_{1}}{\left\vert z_{1}\right\vert }$,
passing through $z_{1}$.

Note that $\big\vert z_{2}-\frac{z_{1}}{\left\vert z_{1}\right\vert }\big\vert
\geq 1-\left\vert z_{2}\right\vert $
and $\big\vert z_{1}-\frac{z_{2}}{\left\vert z_{2}\right\vert }\big\vert
\geq 1-\left\vert z_{1}\right\vert $
whenever $z_{1}\neq 0\neq z_{2}$, by the triangle inequality.

\smallskip

The formulas for $m$ and the above analytical characterizations of $0\in
H_{j}$ and of $\frac{z_{k}}{\left\vert z_{k}\right\vert }\in H_{j}$ for $%
j,k\in \left\{ 1,2\right\} $ imply the claim.

Moreover, $z^{\prime },z^{\prime \prime }$ are the roots \ of the quadratic
equation
\begin{equation*}
  (\overline{z_{1}}-\overline{z_{2}})z^{2}
   -\left( \left\vert z_{1}\right\vert^{2}-\left\vert z_{2}\right\vert^{2}\right)z
   +(z_{1}-z_{2})=0,
\end{equation*}%
as $z\in \left\{ z^{\prime },z^{\prime \prime }\right\} $ implies $L\left(
\frac{1}{z}\right) =L(\overline{z})=\overline{L(z)}=0.$
\end{proof}

\begin{rem} \label{ChoiceZ}
{\rm
  The formula \eqref{intersecmed} is invariant to rotations
  around the origin.

  It follows that
  $\min \left\{ \left\vert z^{\prime }-z_{1}\right\vert,
   \left\vert z^{\prime \prime }-z_{1}\right\vert \right\} =\left\vert z^{\ast}
   -z_{1}\right\vert $, with
  \begin{equation*}
    z^{\ast }=\frac{z_{1}-z_{2}}
                 {\left\vert z_{1}-z_{2}\right\vert }
            \left( \frac{\left\vert z_{1}\right\vert ^{2}-\left\vert
                     z_{2}\right\vert ^{2}}{2\left\vert z_{1}-z_{2}\right\vert }
            +i\cdot {\rm signum}\big( {\Im}\left( \overline{z_{1}}z_{2}\right)
             \big) \vspace{6pt}
            \sqrt{1-\bigg( \frac{\left\vert z_{1}\right\vert^{2
                        }-\left\vert z_{2}\right\vert ^{2}}
                     {2\left\vert z_{1}-z_{2}\right\vert }\bigg) ^{2}}\right) ,
  \end{equation*}%
  where we assume ${\Im}\left( \overline{z_{1}}z_{2}\right) \neq 0.$

  If ${\Im}\left( \overline{z_{1}}z_{2}\right) =0$, i.e. $0,$ $z_{1},$
  $ z_{2}$ are collinear, then  $\left\vert z^{\prime }-z_{1}\right\vert
  =\left\vert z^{\prime \prime }-z_{1}\right\vert $ and we can choose any
  $ z^{\ast }\in \left\{ z^{\prime },z^{\prime \prime }\right\} .$
}
\end{rem}

\section{Barrlund's metric and quasiconformal maps}

In this section we will study how Barrlund's metric behaves under
quasiconformal mappings. 
We first consider the case of M\"obius transformations.

The main property of the hyperbolic metric is its invariance
under the M\"obius self-mapping $T_a: \mathbb{D} \to \mathbb{D}\,, z \mapsto
\frac{z-a}{1- \overline{a}z}\,, |a|<1\,,$ of the unit disk:
\begin{equation*}
  \rho_{\mathbb{D}}(T_a(z_1),T_a(z_2)) =\rho_{\mathbb{D}}(z_1,z_2)
\end{equation*}
for all $z_1,z_2 , a \in \mathbb{D} \,.$ In other words, the mapping $T_a$
is an isometry. Now making use of \eqref{eq:hvz216}, Theorem \ref{lem:monot},
and the properties of the triangular ratio metric, we can prove that $T_a$ is a
Lipschitz mapping with respect to the Barrlund metric. The proof is based on
\cite[Theorem 4.8]{hvz} and the same proof would  also give similar results
for M\"obius transformations between half planes.

\begin{thm}
 Let $p\geq 1$. For $a$, $z_{1}$, $z_{2}\in \mathbb{D}$ we have
  \begin{equation*}
    \displaystyle
    b_{\mathbb{D},p}(T_{a}(z_{1}),T_{a}(z_{2}))\leq 2^{2-\frac{1}{p}}
     \frac{b_{\mathbb{D},p}(z_{1},z_{2})}{1+b_{\mathbb{D},p}(z_{1},z_{2})^{2}}.
  \end{equation*}
\end{thm}

\begin{proof}
  By \cite[Theorem 4.8]{hvz} 
  $s_{\mathbb{D}}(T_{a}(z_{1}),T_{a}(z_{2}))
    \leq 2\frac{s_{\mathbb{D}}(z_{1},z_{2})}{1+s_{\mathbb{D}}(z_{1},z_{2})^{2}}$ 
  and by Theorem \ref{lem:monot} 
  $s_{\mathbb{D}}\leq b_{\mathbb{D},p}\leq 2^{1-\frac{1}{p}}s_{\mathbb{D}}$ 
  on $\mathbb{D}$.
  The claim follows using the fact that $t\mapsto \frac{t}{1+t^{2}}$ is
  increasing on $\left[ 0,1\right] $.
\end{proof}
We give a generalization of \cite[Theorem 3.31]{chkv} for $n=2$, which can be
extended to the case $n\geq 2$.

\begin{thm}
  Let $1\leq p\leq \infty $ and $a\in
  \mathbb{D}$. Then $T_{a}:(\mathbb{D},b_{\mathbb{D},p})\rightarrow 
  (\mathbb{D},b_{\mathbb{D},p})$ is $L$-bilipschitz with 
  $L=\frac{1+\left\vert a\right\vert }{1-\left\vert a\right\vert }$.
\end{thm}

\begin{proof}
For every $u,v\in \mathbb{D}$,
\begin{equation*}
  T_{a}(u)-T_{a}(v)=b\frac{u-v}{(u-a^{\ast })(v-a^{\ast })},
\end{equation*}%
where $a^{\ast }=a/\left\vert a\right\vert^{2}$ and 
$b=(1-\left\vert a\right\vert ^{2})/\overline{a}^{2}$.

Let $z_1,z_2\in \mathbb{D}$ be distinct points. We prove that%
\begin{equation}
  \frac{1-\left\vert a\right\vert }{1+\left\vert a\right\vert }
   b_{\mathbb{D},p}(z_1,z_2)\leq b_{\mathbb{D},p}(T_{a}(z_1),T_{a}(z_2))
  \leq \frac{1+\left\vert a\right\vert }{1-\left\vert a\right\vert }
    b_{\mathbb{D},p}(z_1,z_2).
  \label{pinf}
\end{equation}

If $1\leq p<\infty $, for every $w\in \partial \mathbb{D}$
\begin{align*}
  Q_{p}(z_1,z_2,w)
  &:=\bigg(\frac{\left\vert T_{a}\left( z_1\right) -T_{a}\left( z_2\right)
      \right\vert }{\sqrt[p]{\left\vert T_{a}\left( z_1\right) -T_{a}
       \left( w\right)\right\vert ^{p}+\left\vert T_{a}\left( z_2\right)
       -T_{a}\left( w\right)\right\vert ^{p}}}\bigg) \bigg/
      \bigg(\frac{\left\vert z_1-z_2\right\vert }
     {\sqrt[p]{\left\vert z_1-w\right\vert ^{p}
               +\left\vert z_2-w\right\vert ^{p}}} \bigg) \\
  &=\left( \frac{\left\vert z_1-w\right\vert ^{p}
               +\left\vert z_2-w\right\vert ^{p}}
     {c^{p}\left\vert z_1-w\right\vert ^{p}+d^{p}\left\vert z_2
    -w\right\vert ^{p}}\right) ^{1/p},
\end{align*}%
where $c:=\left\vert z_2-a^{\ast }\right\vert /\left\vert w-a^{\ast
}\right\vert $ 
and $d:=\left\vert z_1-a^{\ast }\right\vert /\left\vert w-a^{\ast }\right\vert $.

Since $\left\vert w-a^{\ast }\right\vert \leq 1+\left\vert a\right\vert ^{-1}
$ and $\left\vert z_1-a^{\ast }\right\vert ,\left\vert z_2-a^{\ast }\right\vert
\geq \left\vert a\right\vert ^{-1}-1$, 
we have $c,d\geq (1-\left\vert a\right\vert )/(1+\left\vert a\right\vert )$. 
Therefore, $Q_{p}(z_1,z_2,w)\leq
\frac{1+\left\vert a\right\vert }{1-\left\vert a\right\vert }=:L$, 
hence
\begin{eqnarray*}
  \frac{\left\vert T_{a}\left( z_1\right) -T_{a}\left( z_2\right) \right\vert }
       {\sqrt[p]{\left\vert T_{a}\left( z_1\right)
          -T_{a}\left( w\right) \right\vert^{p}+\left\vert T_{a}\left( z_2\right)
          -T_{a}\left( w\right) \right\vert ^{p}}}
  &\leq &L\frac{\left\vert z_1-z_2\right\vert }
               {\sqrt[p]{\left\vert z_1-w\right\vert^{p}
            +\left\vert z_2-w\right\vert ^{p}}} \\
  &\leq &Lb_{\mathbb{D},p}(z_1,z_2).
\end{eqnarray*}

As $T_{a}(\partial \mathbb{D)=}\partial \mathbb{D}$, taking supremum
over all $w\in \partial \mathbb{D}$ yields
\begin{equation*}
  b_{\mathbb{D},p}(T_{a}(z_1),T_{a}(z_2))\leq \frac{1+\left\vert a\right\vert }{%
  1-\left\vert a\right\vert }b_{\mathbb{D},p}(z_1,z_2).
\end{equation*}

Having $T_{a}^{-1}=T_{-a}$, it follows similarly that
$b_{\mathbb{D},p}(z_1,z_2)
  \leq \frac{1+\left\vert a\right\vert }{1-\left\vert a\right\vert }
    b_{\mathbb{D},p}(T_{a}(z_1),T_{a}(z_2))$.
Then (\ref{pinf}) holds.

If $p=\infty $, for every $w\in \partial \mathbb{D}$
\begin{align*}
  R(z_1,z_2,w)
  &:=\bigg(\frac{\left\vert T_{a}\left( z_1\right) -T_{a}\left( z_2\right)
      \right\vert }{\max \left\{ \left\vert T_{a}\left( z_1\right)
       -T_{a}\left(w\right) \right\vert ,
       \left\vert T_{a}\left( z_2\right) -T_{a}\left( w\right)
      \right\vert \right\} }\bigg)\bigg/
      \bigg(\frac{\left\vert z_1-z_2\right\vert }{\max \left\{
      \left\vert z_1-w\right\vert ,\left\vert z_2-w\right\vert \right\}
     }\bigg) \\
  &=\frac{\max \left\{ \left\vert z_1-w\right\vert ,\left\vert z_2-w\right\vert
     \right\} }{\max \left\{ c\left\vert z_1-w\right\vert ,d\left\vert
     z_2-w\right\vert \right\} },
\end{align*}
with $c,$ $d$ as above. Then
\begin{eqnarray*}
  \frac{\left\vert T_{a}\left( z_1\right) -T_{a}\left( z_2\right) \right\vert }{%
  \max \left\{ \left\vert T_{a}\left( z_1\right) -T_{a}\left( w\right)
  \right\vert ,\left\vert T_{a}\left( z_2\right) -T_{a}\left( w\right)
  \right\vert \right\} }
  &\leq &L\frac{\left\vert z_1-z_2\right\vert }{\max
   \left\{ \left\vert z_1-w\right\vert ,
    \left\vert z_2-w\right\vert \right\} } \\
  &\leq &Lb_{\mathbb{D},\infty }(z_1,z_2),
\end{eqnarray*}%
hence $b_{\mathbb{D},\infty }(T_{a}(z_1),T_{a}(z_2))\leq \frac{1+\left\vert
a\right\vert }{1-\left\vert a\right\vert }b_{\mathbb{D},\infty }(z_1,z_2)$. As
above, it follows that (\ref{pinf}) also holds for $p=\infty $.
\end{proof}

\begin{conjecture}\label{chkvConj} By the above results we see that
  there exists for $p\in[1,\infty], a \in \mathbb{D}\,,$
  the least constant $R(p,a)$ such that for all  
  $z_1,z_2 \in  \mathbb{D} $
  \begin{equation*}
    b_{\mathbb{D},p}(T_a(z_1), T_a(z_2))
     \le R(p,a) b_{\mathbb{D},p}(z_1,z_2)\,.
  \end{equation*}

On the basis of computer experiments we expect that the
following inequality holds for $p=1,2$
  \begin{equation*}
  R(p,a) \le 1+ |a| \,.
  \end{equation*}
\end{conjecture}

In the case $p=1$ Conjecture \ref{chkvConj}
was formulated in \cite{chkv} and it was shown
in \cite[Thm 1.5]{chkv} that $ R(1,a) \ge 1+ |a|\,.$ We now
extend this last inequality for all $p\,.$

\begin{thm}  
  For all $1\leq p\leq \infty $ and $a\in \mathbb{D}\,$
   $R(p,a)\geq 1+\left\vert a\right\vert \,. $
\end{thm}

\begin{proof} We may assume $a\neq 0$, as $R(p,0)=1$. Denote $\alpha =\arg
\left( -a\right) $. Then $T_{a}(re^{i\alpha })=\frac{r+\left\vert
a\right\vert }{1+r\left\vert a\right\vert }e^{i\alpha }$ for all $r\in
\lbrack 0,1)$.

Let $0\leq r<s<1$. For all $t\in \mathbb{R}$,
\begin{equation*}
b_{\mathbb{D},p}(re^{it},se^{it})=\frac{s-r}{\sqrt[p]{(1-r)^{p}+(1-s)^{p}}}
\end{equation*}%
and $b_{\mathbb{D},\infty }(re^{it},se^{it})=\frac{s-r}{1-r}$.

Note that $0<e^{-i\alpha }T_{a}(re^{i\alpha })<e^{-i\alpha
}T_{a}(se^{i\alpha })<1$.

Assume that $1\leq p<\infty $. Then
\begin{align*}
b_{\mathbb{D},p}(T_{a}(re^{i\alpha }),T_{a}(se^{i\alpha }))
  & =\frac{\frac{s+\left\vert a\right\vert }{1+s\left\vert a\right\vert }
     -\frac{r+\left\vert a\right\vert }{1+r\left\vert a\right\vert }}
      {\sqrt[p]{(1-\frac{r+\left\vert a\right\vert }
                {1+r\left\vert a\right\vert })^{p}
     +(1-\frac{s+\left\vert a\right\vert }{1+s\left\vert a\right\vert })^{p}}}\\
  & =\frac{\left( 1+\left\vert a\right\vert \right) \left( s-r\right) }
             {\sqrt[p]{\left( 1+s\left\vert a\right\vert \right) ^{p}(1-r)^{p}
     +\left( 1+r\left\vert a\right\vert \right)^{p}(1-s)^{p}}}.
\end{align*}
Therefore,
\begin{align*}
  R(p,a)
     &\geq \frac{b_{\mathbb{D},p}(T_{a}(re^{i\alpha }),T_{a}(se^{i\alpha }))}
                {b_{\mathbb{D},p}(re^{it},se^{it})}\\
     &=\left( 1+\left\vert a\right\vert\right)
        \sqrt[p]{\frac{(1-r)^{p}+(1-s)^{p}}{\left( 1+s\left\vert
        a\right\vert \right) ^{p}(1-r)^{p}
       +\left( 1+r\left\vert a\right\vert \right)^{p}(1-s)^{p}}}.
\end{align*}

Similarly, 
$b_{\mathbb{D},\infty }(T_{a}(re^{i\alpha }),T_{a}(se^{i\alpha}))
 =\frac{\frac{s+\left\vert a\right\vert }{1+s\left\vert a\right\vert }-%
     \frac{r+\left\vert a\right\vert }{1+r\left\vert a\right\vert }}
       {1-\frac{r+\left\vert a\right\vert }{1+r\left\vert a\right\vert }}
 =\frac{\left(1+\left\vert a\right\vert \right) 
     \left( s-r\right) }{\left( 1+s\left\vert a\right\vert \right) (1-r)}$, 
hence
\begin{equation*}
  R(\infty ,a)\geq 
   \frac{b_{\mathbb{D},\infty }(T_{a}(re^{i\alpha}),T_{a}(se^{i\alpha }))}
         {b_{\mathbb{D},\infty }(re^{it},se^{it})}
  =\frac{1+\left\vert a\right\vert }{1+s\left\vert a\right\vert }.
\end{equation*}

As $s\rightarrow 0$, it follows that $r\rightarrow 0$ and $R(p,a)\geq
1+\left\vert a\right\vert $ for $1\leq p\leq \infty .$
\end{proof}

 By \cite[Corollary 3.30]{chkv} and Theorem \ref{lem:monot}
(extended to include the case $%
p=\infty $), we obtain

\begin{prop}
  Let $f:G\rightarrow \Omega $ be a M\"{o}bius
  transformation onto $\Omega $, where 
  $G,\Omega \in \left\{ \mathbb{D},\,\mathbb{H}\right\} $ 
  and let $1\leq p\leq \infty $. Then $%
  f:(G,b_{G,p})\rightarrow (\Omega ,b_{\Omega ,p})$ is $L$-Lipschitz
  with $L=2^{2-1/p}$ if $G=\mathbb{D}$, respectively 
  $L=2^{1-1/p}$
  if $G=\mathbb{H}$.
\end{prop}

We also recall some notation about special functions and
the fundamental distortion result of
quasiregular maps, a variant of the Schwarz lemma for these maps.
For $r\in(0,1)$ and $K>0$, we define the distortion function
\[
    \varphi_K(r)=\mu^{-1}(\mu(r)/K),
\]
where $\mu(r)$ is the modulus of the planar Gr\"otzsch ring,
see \cite[pp. 92-94]{avv}, \cite[Exercise 5.61]{vu2}.

\begin{lem}\label{st4}{\rm \cite[Theorem 11.2]{vu2}}
  Let $f:D \to G\,,$  $D,G \in \{  \mathbb{B}^n,   \mathbb{H}^n \}$
  be a non-constant $K$-quasiregular mapping with $f D\subset G$.
  Then for all $z_1,\,z_2\in  D$,
  \[
    {\rm tanh}\, \half \rho_{G}\left(f(z_1),f(z_2)\right)
    \le \varphi_K\left({\rm tanh} \,\half \rho_{D}(z_1,z_2)\right)
    \le 4^{1-1/K} \left({\rm tanh} \,\half \rho_{D}(z_1,z_2) \right)^{1/K}\,\, .
  \]
\end{lem}

\begin{nonsec}{\bf Proof of Theorem \ref{myBarrQc}.} {\rm
By Theorem \ref{lem:monot}  and Lemma \ref{st4}
\begin{align*}
  b_{\mathbb{H},p}(f(z_1), f(z_2))
   & \le 2^{1 -1/p}\,s_{\mathbb{H}}(f(z_1), f(z_2))
      =2^{1 -1/p}\,{\rm tanh} \frac{\rho_{\mathbb{H}}(f(z_1), f(z_2))}{2}\\
   &\le 4^{1-1/K} \,2^{1 -1/p}\,\left({\rm tanh }
          \frac{\rho_{\mathbb{H}}(z_1, z_2)}{2} \right)^{1/K} \\
   & =4^{1-1/K} \,2^{1 -1/p}\,\left(s_{\mathbb{H}}(z_1, z_2) \right)^{1/K}
      \le 4^{1-1/K} \,2^{1 -1/p}\, {b_{\mathbb{H},p}(z_1,z_2) }^{1/K} \,.
\mathqed
\end{align*}
}
\end{nonsec}

\begin{rem} {\rm
  Theorem \ref{myBarrQc} is sharp in the following
  sense. If $p=1\,,$ then the conclusion is
\[
  s_{\mathbb{H}}(f(z_1), f(z_2))
     \le  4^{1-1/K} {s_{\mathbb{H}}(z_1,z_2) }^{1/K}
    \]
  and the constant $ 4^{1-1/K}$ cannot be replaced by any number $c<1\,.$
  Moreover, if $p=2, K=1\,,$ the result says that
  \[
   b_{\mathbb{H},2}(f(z_1), f(z_2))
     \le \sqrt{2}  {b_{\mathbb{H},2}(z_1,z_2) }^{1/K} \,.
     \]
  The constant  $\sqrt{2}$ is sharp, because by numerical experiments
  this constant is attained if $h(x) = x/|x|^2\,,$ which maps $\mathbb{H}$
  onto itself,
  and $z_1= i c, z_2=2 + i t$ where $c>0$ and $t>0$ are close to zero. }
\end{rem}

We generalize \cite[Theorem 4.4]{hvz}, using also some ideas from
\cite[Proposition 2.2]{h2}.

\begin{thm} \label{barpBilip}
  Let $G$, $D\subsetneq \mathbb{R}^{n}$ be domains and 
  $1\leq p<\infty $. 
  Let $f:G\rightarrow D$ be a
  surjective mapping satisfying
  the $L$-bilipschitz condition with respect to the
  $p$-Barrlund metric, for some $L\geq 1$, i.e.
  \begin{equation}
    b_{G,p}(z_1,z_2)/L\leq b_{D,p}(f(z_1),f(z_2))\leq Lb_{G,p}(z_1,z_1)
    \label{pbiLip}
  \end{equation}%
  for all $z_1,z_2\in G$. Then $f$ is a quasiconformal
  homeomorphism (either sense-preserving or sense-reversing), with the linear
  dilatation bounded from above by $4^{1-\frac{1}{p}}L^{2}$.
\end{thm}

\begin{proof}
The first inequality in (\ref{pbiLip}) shows that $f$ is
injective, hence $f$ is bijective.\ We will prove that $f$ is continuous.
Since the inverse $f^{-1}$ also satisfies the $L$-bilipschitz 
condition with
respect to the $p$-Barrlund metric, it will follow that $f^{-1}$ is
continuous, therefore $f$ is a homeomorphism.

Let $z_1,z_2\in G$.

It is easy to see that
\begin{equation}
  b_{G,p}(z_1,z_2)\leq \frac{\left\vert z_1-z_2\right\vert }{\left(
  d_{G}(z_1)^{p}+d_{G}(z_2)^{p}\right) ^{1/p}},  \label{secondb}
\end{equation}
hence, for all $z_1,z_2\in G$,
\begin{equation*}
  \left\vert z_1-z_2\right\vert
  \geq \left( d_{G}(z_1)^{p}+d_{G}(z_2)^{p}\right)^{1/p}b_{G,p}(z_1,z_2).
\end{equation*}
Now let $w\in \partial G$ with $d_{G}(z_1)=\left\vert z_1-w\right\vert $. Then
\begin{equation*}
  b_{G,p}(z_1,z_2)\geq s_{G}\left( z_1,z_2\right)
    \geq \frac{\left\vert z_1-z_2\right\vert}
              {\left\vert z_1-w\right\vert +\left\vert w-z_2\right\vert }.
\end{equation*}%
But
$\left\vert w-z_2\right\vert
 \leq \left\vert z_1-w\right\vert +\left\vert z_1-z_2\right\vert $,
hence
$b_{G,p}(z_1,z_2)
  \geq \frac{\left\vert z_1-z_2\right\vert }
            {2d_{G}\left( z_1\right) +\left\vert z_1-z_2\right\vert }$.
By symmetry, we get as
in \cite{hvz} the stronger inequality
\begin{equation}
  b_{G,p}(z_1,z_2)
  \geq \frac{\left\vert z_1-z_2\right\vert }
            {\left\vert z_1-z_2\right\vert+2\min \left\{ d_{G}\left( z_1\right) ,
                 d_{G}(z_2)\right\} }.  \label{firstb}
\end{equation}
If $0<b_{G,p}(z_1,z_2)<1$ this implies
\begin{equation*}
  \left\vert z_1-z_2\right\vert
    \leq \frac{2\min \left\{ d_{G}\left( z_1\right),d_{G}(z_2)\right\} }
              {\frac{1}{b_{G,p}(z_1,z_2)}-1}.
\end{equation*}

\smallskip

Fix $z\in G$. For every $u\in G\setminus \left\{ z\right\} $ we have $%
f(u)\neq f(z)$ and using inequalities corresponding to (\ref{firstb}) and (%
\ref{secondb}), respectively, we get

\begin{eqnarray}
  1+\frac{2d_{D}(f(z))}{\left\vert f(u)-f(z)\right\vert }
  &\geq &1+\frac{2\min \left\{ d_{D}\left( f(u)\right) ,d_{D}(f(z))\right\} }
                {\left\vert f(u)-f(z)\right\vert }
          \geq \frac{1}{b_{D,p}(f(u),f(z))}  \label{continseq}\\
  &\geq &\frac{1}{Lb_{G,p}(u,z)}\geq \frac{1}{L}\cdot
         \frac{\left(d_{G}(u)^{p}+d_{G}(z)^{p}\right) ^{1/p}}
              {\left\vert u-z\right\vert }
         \geq\frac{1}{L}\frac{d_{G}(z)}{\left\vert u-z\right\vert }.  \notag
\end{eqnarray}
If $0<\left\vert u-z\right\vert <\frac{1}{L}d_{G}(z)$ it follows that $%
0<b_{D,p}(f(u),f(z))<1$ and
\begin{equation*}
  \left\vert f(u)-f(z)\right\vert
    \leq 2Ld_{G}(z)\frac{\left\vert u-z\right\vert }
                       {d_{G}(z)-L\left\vert u-z\right\vert }.
\end{equation*}

We conclude that $f$ is continuous at the arbitrary point $z\in G$.

\medskip

The linear dilatation of the homeomorphism $f$ at $z\in G$ is defined by
\begin{equation*}
  H_{f}(z):=\underset{r\rightarrow 0}{\lim \sup }\frac{L_{f}(z,r)}{l_{f}(z,r)},
\end{equation*}%
where
$
  L_{f}(z,r):=\sup \left\{ \left\vert f(z_1)-f(z)\right\vert 
  :\left\vert 
  z_1-z\right\vert =r\right\} 
$ and
$
  l_{f}(z,r):=\inf \left\{ \left\vert
  f(z_1)-f(z)\right\vert :\left\vert z_1-z\right\vert =r\right\}.
$

If $u\in G$ with $0<\left\vert u-z\right\vert <\frac{1}{L}d_{G}(z)$,
revisiting inequalities (\ref{continseq}) we get
\begin{equation*}
  \left\vert f(u)-f(z)\right\vert
    \leq \frac{2\min \left\{ d_{D}\left(f(u)\right) ,
                        d_{D}(f(z))\right\} }
              {\frac{1}{L}.\frac{\left(d_{G}(u)^{p}+d_{G}(z)^{p}\right)^{1/p}}
                               {\left\vert u-z\right\vert }-1}.
\end{equation*}

On the other hand, for every $v\in G$,
\begin{eqnarray*}
  \left\vert f(v)-f(z)\right\vert
   &\geq &\Big(d_{D}(f(v))^{p}+d_{D}(f(z))^{p}\Big)^{1/p}
                b_{D,p}(f(v),f(z)) \\
   &\geq &\frac{1}{L}\Big(d_{D}(f(v))^{p}
           d_{D}(f(z))^{p}\Big)^{1/p}b_{G,p}(v,z) \\
   &\geq &\frac{1}{L}\Big(d_{D}(f(v))^{p}
          d_{D}(f(z))^{p}\Big) ^{1/p}%
         \frac{\left\vert v-z\right\vert }
              {\left\vert v-z\right\vert
         +2\min \left\{d_{G}\left( v\right) ,d_{G}(z)\right\} }.
\end{eqnarray*}

For every $\varepsilon $ with 
$0<\varepsilon <d_{D}(f(z))$ consider 
$\delta\left( \varepsilon ,z\right) >0$ such that 
$\left\vert f(z_1)-f(z_2)\right\vert <\varepsilon $ 
for every $z_1\in G$ with 
$\left\vert z_1-z\right\vert <\delta(\varepsilon ,z)$.

Let $0<r<\min \left\{ \frac{1}{L}d_{G}(z),\delta \left( \varepsilon
,z\right) \right\} $. Assuming that $\left\vert u-z\right\vert =\left\vert
v-z\right\vert =r$ we obtain from the above inequalities%
\begin{equation*}
  \frac{\left\vert f(u)-f(z)\right\vert }{\left\vert f(v)-f(z)\right\vert }%
  \leq L^{2}\frac{2\min \left\{d_D\left(f(u)\right), d_D(f(z))\right\}}
                {\left(d_D(f(v))^{p}  d_{D}(f(z))^{p}\right) ^{1/p}}
   \frac{2\min \left\{d_{G}\left( v\right) ,d_{G}(z)\right\} +r}
        {\left(d_{G}(u)^{p}+d_{G}(z)^{p}\right) ^{1/p}-Lr}.
\end{equation*}%
Then
\begin{equation*}
  \frac{L_{f}(z,r)}{l_{f}(z,r)}
  \leq L^{2}\frac{2d_D(f(z))}
                {(\left( d_D(f(z))-\varepsilon )^{p}
                  d_{D}(f(z))^{p}\right) ^{1/p}}
          \frac{2d_{G}(z)+r}{(\left( d_{G}(z)-r)^{p}+d_{G}(z)^{p}\right) ^{1/p}-Lr}.
\end{equation*}

As $r$ tends to zero, we conclude that
\begin{equation*}
  H_{f}(z)\leq L^{2}\frac{2^{2-\frac{1}{p}}d_D(f(z))}{(\left(
  d_D(f(z))-\varepsilon )^{p} d_{D}(f(z))^{p}\right) ^{1/p}},
\end{equation*}%
hence letting $\varepsilon \rightarrow 0$ it follows that $H_{f}(z)\leq 4^{1-%
\frac{1}{p}}L^{2}$.
\end{proof}

As expected, the above result has a counterpart in the case $p=\infty $.

\begin{thm}
  Let $G$, $D\subsetneq \mathbb{R}^{n}$ be domains and
  let $f:G\rightarrow D$ be a surjective mapping satisfying
  the $L$-bilipschitz condition with respect
  to the $\infty $-Barrlund metric, for some $L\geq 1$,
  i.e.
  \begin{equation}
     b_{G,\infty }(z_1,z_2)/L\leq b_{D,\infty }(f(z_1),f(z_2))
     \leq Lb_{G,\infty }(z_1,z_2)
     \label{infinitybiLip}
  \end{equation}%
  for all $z_1,z_2\in G$. Then $f$ is a quasiconformal homeomorphism
  (either sense-preserving or sense-reversing), with the linear
  dilatation bounded from above by $4L^{2}$.
\end{thm}

\begin{proof}
Clearly, $f$ is a bijection.
For every $z_1,z_2\in G$,
\begin{equation*}
  \frac{\left\vert z_1-z_2\right\vert }{\left\vert z_1-z_2\right\vert
  +2\min \left\{d_{G}\left( z_1\right) ,d_{G}(z_2)\right\} }
  \leq b_{G,\infty }(z_1,z_2)
  \leq \frac{\left\vert z_1-z_2\right\vert }
            {\max \left\{ d_{G}\left( z_1\right),d_{G}(z_2)\right\} }.
\end{equation*}
If $0<b_{G,\infty }(z_1,z_2)<1$ then
\begin{equation*}
  \left\vert z_1-z_2\right\vert
  \leq \frac{2\min \left\{ d_{G}\left( z_1\right),d_{G}(z_2)\right\} }
            {\frac{1}{b_{G,\infty }(z_1,z_2)}-1}.
\end{equation*}%
\smallskip
Fix $z\in G$. For every $u\in G\setminus \left\{ z\right\} $ we have $%
f(u)\neq f(z)$ and

\begin{eqnarray*}
  1+\frac{2d_{D}(f(z))}{\left\vert f(u)-f(z)\right\vert }
  &\geq &1+\frac{2\min \left\{d_{D}\left( f(u)\right), d_{D}(f(z))\right\} }
                {\left\vert f(u)-f(z)\right\vert }
         \geq \frac{1}{b_{D,\infty }(f(u),f(z))} \\
  &\geq &\frac{1}{Lb_{G,\infty }(u,z)}
         \geq \frac{1}{L}
              \frac{\max \left\{d_{G}\left( u\right) ,d_{G}(z)\right\} }
                   {\left\vert u-z\right\vert }
         \geq\frac{1}{L}\frac{d_{G}(z)}{\left\vert u-z\right\vert }.
\end{eqnarray*}%
As in the proof of Theorem \ref{barpBilip}, the continuity of $f$ follows.
Moreover,  $f^{-1}$ is continuous on $D$.
If $0<\left\vert u-z\right\vert <\frac{1}{L}d_{G}(z)$ it follows that 
$ 0<b_{D,\infty }(f(u),f(z))<1$ and
\begin{equation*}
  \left\vert f(u)-f(z)\right\vert
  \leq \frac{2\min \left\{ d_{fG}\left(f(u)\right) ,d_{fG}(f(z))\right\} }
          {\frac{1}{L}\cdot
           \frac{\max \left\{d_{G}\left( u\right) ,d_{G}(z)\right\} }
                {\left\vert u-z\right\vert }-1}.
\end{equation*}
For every $v\in G$,
\begin{eqnarray*}
  \left\vert f(v)-f(z)\right\vert
  &\geq &\max \left\{d_{D}\left(f(v)\right),d_{D}(f(z))\right\}
              b_{D,\infty }(f(v),f(z)) \\
  &\geq &\frac{1}{L}\max \left\{d_{D}\left( f(v)\right),
             d_{D}(f(z))\right\}
         b_{G,\infty }(v,z) \\
  &\geq &\frac{1}{L}\max \left\{d_{D}\left( f(v)\right), d_{D}(f(z))\right\}
         \frac{\left\vert v-z\right\vert }{\left\vert v-z\right\vert
         +2\min \left\{ d_{G}\left( v\right) ,d_{G}(z)\right\} }.
\end{eqnarray*}
If  $0<r<\frac{1}{L}d_{G}(z)$ and  $\left\vert u-z\right\vert =\left\vert
v-z\right\vert =r$, the latter inequalities yield

\begin{equation*}
  \frac{\left\vert f(u)-f(z)\right\vert }{\left\vert f(v)-f(z)\right\vert }%
  \leq L^{2}\frac{2\min \left\{d_{D}\left( f(u)\right) ,
                d_{D}(f(z))\right\}}
         {\max \left\{d_{D}\left( f(v)\right),d_{D}(f(z))\right\} }
   \frac{2\min
  \left\{ d_{G}\left( v\right) ,d_{G}(z)\right\} +r}{\max \left\{ d_{G}\left(
  u\right) ,d_{G}(z)\right\} -Lr}.
\end{equation*}%
Then
\begin{equation*}
  \frac{L_{f}(z,r)}{l_{f}(z,r)}\leq 2L^{2}\, \frac{2d_{G}(z)+r}{d_{G}(z)-Lr},
\end{equation*}%
hence $H_{f}(z)\leq 4L^{2}.$
\end{proof}

\section*{Acknowledgements}
 This work was partially supported by JSPS KAKENHI
    Grant Number 19K03531 and by JSPS Grant BR171101.



\end{document}